\documentclass[a4paper, 11pt, headings = small, abstract]{scrartcl}

\usepackage[english]{babel}
\usepackage[utf8]{inputenc}
\usepackage{geometry}

\usepackage[semibold]{libertine}
\usepackage[upint,amsthm]{libertinust1math}

\usepackage{amsmath}
\usepackage{amsfonts}
\usepackage{amssymb}
\usepackage{amsthm}
\usepackage{mathtools}
\usepackage{paralist}

\usepackage[authoryear]{natbib}
\usepackage{color}
\usepackage[colorlinks=true,linkcolor=blue,citecolor=blue,pdfborder={0 0 0}]{hyperref}

\usepackage{dsfont}
\usepackage{bm}

\usepackage{graphicx}
\usepackage{enumerate}

\newtheorem{theorem}{Theorem}[section]
\newtheorem{proposition}[theorem]{Proposition}
\newtheorem{lemma}[theorem]{Lemma}
\newtheorem{corollary}[theorem]{Corollary}

\theoremstyle{definition}

\newtheorem{condition}[theorem]{Condition}

\theoremstyle{remark}
\newtheorem{remark}[theorem]{Remark}
\newtheorem{example}[theorem]{Example}

\newtheorem*{lemma-no-number}{Lemma}

\allowdisplaybreaks[4]

\numberwithin{equation}{section}

\newcommand{\R}{\mathbb{R}}

\newcommand{\N}{\mathbb{N}}
\newcommand{\Prob}{\mathbb{P}}    

\newcommand{\Cb}{\mathbb {C}}

\newcommand{\Ic}{\mathcal{I}}
\newcommand{\Mc}{\mathcal{M}}
\newcommand{\Nc}{\mathcal{N}}

\newcommand{\eps}{\varepsilon}

\newcommand{\diff}{{\,\mathrm{d}}}

\newcommand{\Exp}{\operatorname{E}}
\newcommand{\Var}{\operatorname{Var}}
\newcommand{\Cov}{\operatorname{Cov}}

\newcommand{\sinc}{\operatorname{sinc}}

\newcommand{\as}{{\mathrm{a.s.}}}

\begin{document}

\title{\fontsize{16}{19} The empirical copula process in high dimensions: Stute's representation and applications}

\author{
Axel B\"ucher\thanks{Ruhr-Universität Bochum, Fakultät für Mathematik. Email: \href{mailto:axel.buecher@rub.de}{axel.buecher@rub.de}}
\and
Cambyse Pakzad\thanks{Université Paris Nanterre, Modal'X. Email: \href{mailto:cambyse.pakzad@parisnanterre.fr}{cambyse.pakzad@parisnanterre.fr}}
}

\date{\today}

\maketitle

\begin{abstract}
The empirical copula process, a fundamental tool for copula inference, is studied in the high dimensional regime where the dimension is allowed to grow to infinity exponentially in the sample size. Under natural, weak smoothness assumptions on the underlying copula, it is shown that Stute's representation is valid in the following sense: all low-dimensional margins of fixed dimension of the empirical copula process can be approximated by a functional of the low-dimensional margins of the standard empirical process, with the almost sure error term being uniform in the margins. The result has numerous potential applications, and is exemplary applied to the problem of testing pairwise stochastic independence in high dimensions, leading to various extensions of recent results in the literature: for certain test statistics based on pairwise association measures, type-I error control is obtained for models beyond mutual independence.  Moreover, bootstrap-based critical values are shown to yield strong control of the familywise error rate for a large class of data generating processes.
\end{abstract}


\smallskip



\section{Introduction}

Statistically assessing dependence in high dimensions 
has recently attracted a lot of attention, see, for instance, 
\cite{CheZhaZho10,
LiuHan12, 
HanCheLiu17, 
LeuDrt18, 
Yao18, 
Drt20,
XiaLi2021,
Det24,
LiWangYao24}, 
among others. 
To the best of our knowledge, the issue has not been approached by general, high-level nonparametric copula methods yet, which is quite surprising in view of the success of copula methods in the low-dimensional case where the dimension $d$ is fixed  (the only exception we are aware of is a specific approach for testing stochastic independence in high dimensions in \cite{BucPak24}). A central object in this regard is the empirical copula and its standardized estimation error, the empirical copula process, which has played a fundamental role for constructing sophisticated inference procedures on copulas in the low-dimensional regime. For instance, it has been used for testing for stochastic independence \citep{Deh79, GenRem04, GenQueRem07} or more general shape constraints \citep{GenNesQue11, BeaSeo20}, for constructing estimators for dependence measures and functions \citep{GenSeg09, Sch10, ChaFouNes20, BucGen23}, for goodness-of-fit testing \citep{GenRemBea09, KojYan11, Rem17} or for change-point analysis \citep{BucKojRohSeg14}, among others. The present paper aims to enable similar applications in the high-dimensional case by providing and applying suitable limit result on the empirical copula process under weak smoothness assumptions on the copula. 
Presumably, the results will be useful for studying methods related to applications in the areas of finance, genomics or medicine, among others:
for instance, \cite{AnatolyevPyrlik2022} have used Gaussian and $t$-copula models for portfolio allocation using data for $d=3,600$ stocks collected on $n=120$ trading days; \cite{MulCza19} suggested vine copula models for $d=2,100$ stocks; \cite{SahinCzado2024} employed vine copula regression techniques in genomics with $d\approx 500,000$ and sample size $n \approx 500$; \cite{XiaLi2021} suggested a copula-based variable screening approach with $d \approx 30,000$ and $n=120$ and \cite{KasaBhattacharyaRajan2019} applied clustering methods in medicine based on Gaussian mixture copulas with $n \approx 100$ and $d$ up to 2,000.

Let $\bm X=(X_1, \dots, X_d) \in \R^d$ denote a $d$-variate random vector with common cumulative distribution function (cdf) $F$ and continuous marginal cdfs $F_1,\ldots,F_d$. By Sklar's theorem, there exists a unique copula $C:[0,1]^d \to [0,1]$ such that
$
F(\bm x) = C(F_1(x_1), \dots, F_d(x_d))
$
for all $\bm x=(x_1, \dots, x_d) \in \R^d$. 
Let $\bm X_1, \dots, \bm X_n$ denote an i.i.d.\ sample from $\bm X$. The empirical copula, the most common nonparametric estimator of $C$, is defined as
\[
\hat C_n(\bm u) =
\frac1n \sum_{i=1}^n \bm 1(\hat {\bm U}_i \le \bm u) 
=
\frac1n \sum_{i=1}^n \prod_{j=1}^d \bm 1(\hat U_{ij} \le u_j), \quad  \bm u \in [0,1]^d,
\]
where $\hat{\bm U}_i = (\hat U_{i1}, \dots \hat U_{id})$ and $
\hat U_{ij} = \frac1n R_{ij}$ for $j\in\{1, \dots, d\}$. Here,  
$R_{ij}$ denotes the (max-)rank of $X_{ij}$ among $X_{1j}, \dots, X_{nj}$. 
The rescaled estimation error of the empirical copula,
\[
\Cb_n(\bm u) = \sqrt n  \{ \hat C_n(\bm u) - C(\bm u) \}, \quad \bm u \in [0,1]^d,
\]
is called the empirical copula process; it is the main object of interest in this paper.

In all of the statistical applications in the low-dimensional regime mentioned in the first paragraph, a central ingredient is the functional weak convergence of the empirical copula process in the space $\ell^\infty([0,1]^d)$ of bounded real-valued functions defined on $[0,1]^d$, equipped with the supremum distance. Such a result is well-known in the case where $d$ is fixed and some suitable smoothness assumptions on $C$ are met (see \cite{Rus76, FerRadWeg04, Tsu05, Seg12}, among others), and is typically obtained by a linearization of the empirical copula process. More precisely, let $\alpha_n(\bm u) = n^{-1/2} \sum_{i=1}^n \{ \bm 1(\bm U_i \le \bm u) - C(\bm u) \}$ denote the standard empirical process based on the (unobservable) sample $\bm U_1, \dots, \bm U_n$ (with $\bm U_i=(U_{i1}, \dots, U_{id})$ and $U_{ij}=F_j(X_{ij})$), and let $\alpha_{nj}(u_j)=\alpha_n(1, \dots, 1, u_j, 1, \dots, 1)$ denote its $j$th margin; here, $u_j$ appears at the $j$th entry. Then, writing 
\begin{align} \label{eq:barcn}
\bar \Cb_n(\bm u) = \alpha_{n}(\bm u) - \sum_{j=1}^d \dot C_j(\bm u) \alpha_{nj}(u_j),
\end{align}
where we assume that the $j$th first-order partial derivative $\dot C_j(\bm u) = (\partial/{\partial u_j})C(\bm u)$ exists and is continuous on $V_j=\{\bm u \in [0,1]^d:u_j \in (0,1)\}$, we have, for from Proposition 3.1 in \cite{Seg12}, 
\begin{align} \label{eq:cnbcn}
\sup_{\bm u \in [0,1]^d} | \Cb_n(\bm u) - \bar \Cb_n(\bm u)| = o_\Prob(1)
\end{align}
for $n\to\infty$. In fact, under slightly stronger conditions on the smoothness of $C$, the $o_\Prob(1)$-term can be replaced by the almost sure term $O_{\as}(n^{-1/4} (\log n)^{1/2} (\log \log n)^{1/4})$, see Proposition 4.2 in \cite{Seg12}; this almost sure linearization of the empirical copula process is also known as Stute's representation \citep{Stu84}. Note that the result in \eqref{eq:cnbcn} essentially allows ranks to be removed ($\Cb_n$ involves ranks while $\bar \Cb_n$ is rank-free) and thus enables further analysis of $\Cb_n$ or functionals thereof based on the abundance of available limit results/concentration inequalities etc.\ for sums and empirical processes of independent summands. In particular, we readily obtain weak convergence of $\Cb_n$ to a centered Gaussian process \citep[ Proposition 3.1]{Seg12}.

In this paper, we are interested in generalizations of \eqref{eq:cnbcn} for the high-dimensional regime. We will consider both a non-asymptotic `in probability' version (Theorem~\ref{theo:main-new}) as well as an almost sure version for the case where $d=d_n$ may depend on $n$ and actually be larger than~$n$. In the latter case, $(\bm X_1, \dots, \bm X_n)=(\bm X_{1}^{(n)}, \dots, \bm X_{n}^{(n)})$ is a row-wise independent triangular array with copula $C^{(n)}$; we suppress the upper index for notational convenience. Being interested in almost sure error bounds, we then assume that all random variables are defined on the same probability space -- all results apply even if this is not the case, with `almost sure'-rates replaced by `in probability'-rates.  Our main result states that, under some weak smoothness assumptions on $C$ and if $d=d_n$ satisfies $\log d =o( n^{1/3})$, we have, for any fixed $k \in  \N_{\ge 2}$,
\begin{align} \label{eq:cbncn-d}
\sup_{\bm u \in W_k} |\Cb_n(\bm u)- \bar \Cb_n(\bm u)| =
O_\as(n^{-1/4} (\log (nd))^{3/4}) =o_\as(1),
\end{align}
where $W_k = W_{d,k} = \{ \bm u \in [0,1]^d: |\{j: u_j < 1\}| \le k\}$ denotes the set of vectors in $[0,1]^d$ for which at most $k$ coordinates are strictly smaller than 1. Equivalently, 
\[
\max_{I \subset \{1, \dots, d\}: |I| = k} \sup_{\bm w \in [0,1]^k} | \Cb_{n,I}(\bm w) - \bar \Cb_{n,I}(\bm w) | =O_\as(n^{-1/4} (\log (nd))^{3/4}) ,
\]
where, for some function $f$ on $[0,1]^d$ and $I\subset \{1, \dots, d\}$ and  $\bm w =(w_j)_{j \in I} \in [0,1]^I$, we write $f_I(\bm w) := f(\bm w^I)$, where $\bm w^I \in [0,1]^d$ has $j$th component
$w_j^I =  w_j \bm 1(j\in I) + \bm 1(j\notin I)$; note that $C_I$ is the $I$-margin of $C$. Hence, \eqref{eq:cbncn-d} essentially provides a uniform linearization of the empirical copula process over all $k$-variate margins. Similar as in the low dimensional regime, the result allows to apply distributional approximation results for independent summands in high dimensions  (see \cite{Che13, Che22} among others) 
to the empirical copula process and functionals thereof. 

As mentioned above, there are many potential applications of \eqref{eq:cbncn-d}, and we exemplary work out details on pairwise dependence measures and high-dimensional independence tests in Section~\ref{sec:app}. First of all, we will reconsider test statistics for pairwise independence proposed in \cite{HanCheLiu17}. While these authors derive weak limit results for their test statistics for mutual independence only (i.e., $C(\bm u)=\prod_{j=1}^d u_j =: \Pi_d(\bm u)$ for all $\bm u \in [0,1]^d$), we obtain the same limit distribution under much weaker assumptions that only concern the four-dimensional margins of $C$. As a by-product, we provide uniform linearizations of common pairwise dependence measures under general weak smoothness assumptions on $C$. Next, we also propose a bootstrap-based version of the test in \cite{HanCheLiu17} which is shown to hold its level under very general assumptions on $C$ that only involve pairwise dependencies. In fact, it is shown that the new test even provides (asymptotic) strong control of a certain family-wise error rate. Finally, we also consider max-type test statistics for independence based on the Moebius transform of the empirical copula process (see \cite{GenRem04} and \cite{BucPak24} for the low- and high-dimensional regime, respectively), and provide a weak limit result for the case $C=\Pi_d$.

The remaining parts of this paper are organized as follows: the main result on Stute's representation for the empirical copula process in high dimensions is provided in Section~\ref{sec:stute}, with some key lemmas for its proof and a non-asymptotic version postponed to Section~\ref{sec:proofs}. The applications mentioned in the previous paragraph are worked out in Subsections~\ref{sec:app1}-\ref{sec:app3}, with corresponding proofs postponed to Section~\ref{sec:proofapp}. Some less central proofs are postponed to an appendix.

Some notation: for $(a_n)_{n \in \N}, (b_n)_{n \in \N} \subset (0,\infty)$, we write $a_n \ll b_n$ if $\lim_{n \to \infty} a_n/b_n=0$ and $a_n\lesssim b_n$ if there exists some constant $c>0$ independent of $n$ such that $a_n\leq c b_n$ for $n$ large enough. For $I\subset \{1, \dots, d\}$ and $\bm u \in [0,1]^d$, we write $\bm u_I = (u_j)_{j\in I} \in [0,1]^I$. Further recall that, for $\bm w \in [0,1]^I, \bm w^I \in [0,1]^d$
is defined to have $j$th component
$w_j^I =  w_j \bm 1(j\in I) + \bm 1(j\notin I)$. In particular, we may write $\alpha_{nj}(u_j)=\alpha_n(u_j^{\{j\}})$. For $\bm u \in [0,1]^d$, $I_{\bm u}$ denote the set of indexes $j$ such that $u_j<1$; note that $|I_{\bm u}| \le k$ for $\bm u \in W_k$. The independence copula in dimension $k\in\N_{\ge 2}$ is denoted $\Pi_k$, i.e., $\Pi_k(\bm u) = \prod_{i=1}^k u_k$ for $\bm u \in [0,1]^k$. For $k \le d$ with $k,d\in\N$, we write $\mathcal I_k(d)=\{I \subset \{1, \dots, d\}: |I|=k\}$.

\section{Main result: Stute's representation in high dimensions}
\label{sec:stute}

It is instructive to start by recapitulating the assumptions needed for deriving \eqref{eq:cnbcn} in the low-dimensional regime.

\begin{condition}[First-order regularity, Condition 2.1 in \cite{Seg12}] \label{cond:c1}
For each $j\in \{1,\ldots, d\}$, the $j$-th first-order partial derivative of $C$ exists and is continuous on the set $V_j=V_{d,j} = \{ \bm u \in [0,1]^d: u_j \in  (0,1)\}$.
\end{condition}

Here, the $j$th first-order partial derivative of $C$ is given by $
\dot C_j(\bm u):=\lim_{h\to 0}  \{ C(\bm u+ h \bm e_j)-C(\bm u)\} / {h}$
for $\bm u \in V_{j}$, where $\bm e_j$ denotes the $j$th unit vector in $\R^d$. To facilitate notation, 
we let
$\dot C_j(\bm u) = 
\limsup_{h \downarrow 0} \{ C(\bm u+ h \bm e_j)-C(\bm u)\} / {h}$ if $u_j = 0$ and $\dot C_j(\bm u) =
\limsup_{h \downarrow 0} \{ C(\bm u- h \bm e_j)-C(\bm u)\} / {(-h)}$ if $u_j = 1$,
such that $\bar \Cb_n$ in \eqref{eq:barcn} is well-defined. In the low-dimensional regime where $d$ is fixed, Condition~\ref{cond:c1} is sufficient for \eqref{eq:cnbcn}, see Proposition 3.1 in \cite{Seg12}. For an almost sure error bound in that regime, a slightly stronger assumption is needed.

\begin{condition}[Second-order regularity, Condition 4.1 in \cite{Seg12}]  \label{cond:c2-segers} 
Condition~\ref{cond:c1} is met.
Moreover, for each $i,j\in \{1,\ldots, d\}$, the second-order partial derivative $\ddot C_{ij}$ exists and is continuous on the set $V_{i}\cap V_{j}$. There exists a constant $K$ such that, for any $i,j\in \{1, \dots, d\}$, 
\begin{align*}
|\ddot C_{ij}(\bm u)| &\leq K \min \left( \frac{1}{u_j(1-u_j)} , \frac{1}{u_i(1-u_i)}   \right) \qquad \forall\ \bm u\in V_{i}\cap V_{j}.
\end{align*}
\end{condition}

In the low-dimensional regime, Condition~\ref{cond:c2-segers} implies \eqref{eq:cnbcn} with the $o_\Prob(1)$-error term replaced by $O_{\as}(n^{-1/4} (\log n)^{1/2} (\log \log n)^{1/4})$, see Proposition 4.2 in \cite{Seg12}. It is worthwhile to mention that the condition is known to hold for many common copula families like the Gaussian copula or many extreme-value copulas; see Section 5 in \cite{Seg12} and Example~\ref{ex:copulas} below.

Next, in the general regime where $d=d_n$ is allowed to depend on $n$, we will need a version of Condition~\ref{cond:c2-segers} with uniform control over the explosive behavior of the second-order partial derivatives near the boundaries, uniform over all $k$-variate margins.  Recall the set $W_k$ of vectors in $[0,1]^d$ for which at most $k$ coordinates are strictly smaller than 1:
\begin{align} \label{eq:Wk}
W_k = W_{d,k} = \{ \bm u \in [0,1]^d: |\{j: u_j < 1\}| \le k\}.
\end{align}

\begin{condition}[Second-order regularity of $k$-variate margins]\label{cond:c2-new} 
Fix $k \in \N_{\ge 2}$ with $k\le d$. For each $j \in \{1, \dots, d\}$, the $j$th first-order partial derivative $\dot C_j$ exists and is continuous on $V_{j} \cap W_k$. 
Moreover, for each $i,j\in \{1,\ldots, d\}$, the second-order partial derivative $\ddot C_{ij}$ exists and is continuous on the set $V_{i}\cap V_{j} \cap W_k$. There exists a constant $K$ such that, for any $i,j\in \{1, \dots, d\}$,
\begin{align*}
|\ddot C_{ij}(\bm u)| &\leq K \min \left( \frac{1}{u_j(1-u_j)} , \frac{1}{u_i(1-u_i)}   \right)   \qquad \forall\ \bm u\in V_{i}\cap V_{j} \cap W_k.
\end{align*}
\end{condition}

\begin{remark}
If $k = d$, we have $W_d=[0,1]^d$, whence Condition~\ref{cond:c2-new} is equivalent to Con\-di\-tion~\ref{cond:c2-segers}. If $k<d$, Condition~\ref{cond:c2-new} is equivalent to requiring that Condition~\ref{cond:c2-segers} is met for any $k$-variate margin $C_I$ (see the arguments below \eqref{eq:cbncn-d}), with the constant $K$ being uniform in $I$. 
\end{remark}

Note that $\bar \Cb_n(\bm u)$ from \eqref{eq:barcn} is well-defined for any $\bm u \in W_k$ under Condition~\ref{cond:c2-new}. Indeed, since $\alpha_{nj}(1)=0$, it can be written as $\bar \Cb_n(\bm u)= \alpha_n(\bm u) - \sum_{j \in I_{\bm u}} \dot C_j(\bm u) \alpha_{nj}(u_j)$, where $I_{\bm u}$ denotes the set of indexes for which $u_j<1$. 
Our main result is as follows; it is proven in Section~\ref{sec:proofs}, where we derive it from a non-asymptotic `in probability'-version that is of independent interest.

\begin{theorem}[Stute's representation in high dimensions] \label{theo:main}
Fix $k\in\N_{\ge 2}$ and suppose that Condition~\ref{cond:c2-new} is met for this choice of $k$ and for any $C=C^{(n)}$, with $K=K(k)$ being independent of $n$ and $C$. Then, if $d$ satisfies $\log d = o(n^{1/3})$,  we have
\[
\sup_{\bm u \in W_k} |\Cb_n(\bm u)- \bar \Cb_n(\bm u)| =O_\as(n^{-1/4} (\log (nd))^{3/4} ) =o_\as(1)
\]
with $W_k$ from \eqref{eq:Wk}.
The implicit constant in the error bound only depends on $k$ and $K$. 
\end{theorem}

Obviously, the assumptions allow for the case $d_n\equiv d= k$ with $k$ fixed. The final rate is then slightly worse than the one in Proposition 4.2 in \cite{Seg12}, with $(\log n)^{3/4}$ instead of $(\log n)^{1/2} (\log \log n)^{1/4}$. This should not have any practical consequences for statistical applications.
Proofs for the following example are provided in Appendix~\ref{sec:example-proofs}.

\begin{example} \label{ex:copulas}
\begin{compactenum}[(a)]
    \item The multivariate Gaussian copula with full-rank  correlation matrix $\Sigma=(\rho_{j,\ell})_{j,\ell=1, \dots d}$, that is, $C(\bm u) = \Phi_\Sigma( x_1, \dots, x_d)$
where $x_j=\Phi^{-1}(u_j)$ with $\Phi$ the standard normal cdf and where $\Phi_\Sigma$ is the cdf of the $\mathcal N_d(\bm 0, \Sigma)$ distribution, forms the basis for graphical models with non-normal margins \citep{LiuLafWas09, LiuHan12}.   
     If $\rho_0:=\max_{j\ne\ell} |\rho_{j,\ell}| <1$, then Condition~\ref{cond:c2-new} is met for any $2 \le k\le d$ with $K=(k-1) [ \rho_0^2/(1-\rho_0^2) ]^{1/2}$. 
    \item The multivariate Hüsler-Reiss copula with parameter matrix $\Lambda =(\lambda_{j,\ell})_{j,\ell=1}^d$, see \cite{Eng15} for details, has recently been extensively used for extremal graphical modelling \citep{Eng24}. If  $\lambda_0:=\min_{j\ne \ell} \lambda_{j,\ell}>0$, then Condition~\ref{cond:c2-new} is met for $k=2$, with some constant $K=K(\lambda_0)$. Recall that $\lambda_{j,\ell} \in [0,\infty)$ parametrizes the bivariate $(j,\ell)$-margin, and interpolates between independence and complete dependence for $\lambda_{j,\ell}=\infty$ and $\lambda_{j,\ell}=0$, respectively.  
    \item The $d$-variate Clayton copula with parameter $\theta\ge 0$, $C_\theta(\bm u) = (u_1^{-\theta} + \dots +u_d^{-\theta} - d + 1)^{-1/\theta}$, satisfies Condition~\ref{cond:c2-new} for any $2 \le k\le d$ with $K=\theta+1$. As the proofs show, similar results can be expected for other Archimedean copulas.
\end{compactenum}
\end{example}

\section{Applications}\label{sec:app}

\subsection{Distributional approximations for maximal pairwise association measures}
\label{sec:app1}

Let $(X,Y)$ denote a bivariate random vector with continuous marginal cdfs and copula $D$. Common association measures for $X$ and $Y$ are given by Spearman's rho, Kendall's tau or Blomquist's beta, which can be expressed in terms of the copula \citep{Sch10}:
$
\rho = 12 \int D \diff \Pi_2 - 3, 
\tau = 4 \int D \diff D - 1$ and $
\beta  = 4 D(\tfrac12, \tfrac12) - 1,
$
where $\Pi_2$ denotes the bivariate indepedence copula. Note that all three association measures are zero if $D$ is the independence copula.

In the high-dimensional regime, for each $I\subset\{1, \dots, d\}$ with $I=2$, let
\begin{align*}
\rho_I = 12 \int C_I \diff \Pi_2 - 3, \qquad
\tau_I = 4 \int C_I \diff C_I - 1, \qquad
\beta_I  = 4 C_I(\tfrac12, \tfrac12) - 1.
\end{align*}
Writing $I=\{\ell,m\}$, the respective sample versions are given by
\begin{align*}
\hat \rho_{n,I} 
&= 
1- \frac{6 \sum_{i=1}^n (R_{i\ell} - R_{im})^2}{n(n-1)(n+1)} 
\\
\hat \tau_{n,I} 
&= \frac2{n(n-1)} \sum_{1 \le i < j \le n} \mathrm{sgn}(X_{i\ell}-X_{j\ell}) \mathrm{sgn}(X_{im}-X_{jm}) \\
\hat \beta_{n,I}  
&= 4 \hat C_{n,I}(\tfrac12, \tfrac12) - 1.
\end{align*}
Recently, \cite{HanCheLiu17, Drt20} studied tests for pairwise independence based on max-type test-statistics like $T_n^{\rho} = \sqrt n \max_{I \in \mathcal I_2(d)} |\hat \rho_{n,I}|$ and  $T_n^\tau = \sqrt n \max_{I \in \mathcal I_2(d)} |\hat \tau_{n,I}|$, where $\mathcal I_2(d)=\{I \subset \{1, \dots, d\}: |I|=2\}$ and where $d=d_n$ grows to infinity. Their derivation of the test statistics' asymptotic behavior under the specific null hypothesis of global independence was based on U-statistics theory. We will provide an alternative approach based on Theorem~\ref{theo:main}, which does not require global independence and which also allows to study $T_n^{\beta} = \sqrt n \max_{I \in \mathcal I_2(d)} |\hat \beta_{n,I}|$; note that $\hat \beta_{n,I}$ is not a U-statistic.
For $I=\{\ell, m\}$, define
\begin{align}
g_I^\rho(\bm u) 
&= \label{eq:rhogi}
12(1-u_\ell)(1-u_m) - 36 \int \Pi_2 \diff C_I + 12\int_0^1 C_I(u_\ell,z)+C_I(z,u_m) \diff z, \\
g_I^\tau(\bm u) 
&=  \nonumber
8C(u_{\ell}, u_{m}) - 4u_{\ell}-4u_{m} + 2 - 2\tau_I,\\
g_I^\beta(\bm u) 
&= \nonumber
4 \Big[ \bm 1(u_{\ell} \le \tfrac12, u_{m} \le \tfrac12) - C_I(\tfrac12, \tfrac12)  \\ \nonumber
& \hspace{3.5cm} - \dot C_\ell(\tfrac 12, \tfrac12) \{ \bm 1(u_{\ell} \le \tfrac12)-\tfrac12\} - \dot C_m(\tfrac 12, \tfrac12) \{ \bm 1(u_{m} \le \tfrac12)-\tfrac12\} \Big],
\end{align}
and let 
$
S_{n,I}^{\gamma}=\frac1{\sqrt n} \sum_{i=1}^n g_I^\gamma(\bm U_i)$ with $\gamma \in \{\rho, \tau, \beta\}.
$
We then have the following result.

\begin{proposition}[Uniform linearization of pairwise association measures] \label{prop:asso}
Let $\gamma \in \{\rho, \tau, \beta\}$. Under the conditions of Theorem~\ref{theo:main} with $k=2$, we have
\begin{align} \label{eq:association}
\max_{I \in \mathcal I_2(d)} |\sqrt{n} \{ \hat \gamma_{n,I} - \gamma_I\} - S_{n,I}^\gamma| = O_\as(n^{-1/4}(\log(nd))^{3/4}) =o_\as(1).
\end{align}
\end{proposition}

The proof is given in Section~\ref{sec:proofapp}. Note that the linearization for Kendall's tau corresponds to the leading term in the Hoeffding decomposition when treating $\hat \tau_{n,I}$ as a U-statistic of order~2.

Based on Proposition~\ref{prop:asso}, we may easily derive distributional approximation results for 
\[
T_n^\gamma=\sqrt n \max_{I \in \mathcal I_2(d)}|\hat \gamma_{n,I}|,
\]
which, as mentioned before, may be regarded as a test statistic for testing the null hypothesis $H_0:(C_I=\Pi_2 \text{ for all } I \in \mathcal I_2(d))$; 
see \cite{HanCheLiu17, Drt20} for details (note that $H_0=H_0^{(n)}$, which we suppress from the notation). The latter authors provide distributional approximation results from which they derive critical values that guarantee asymptotic control of the type-I error in the case where $d$ grows to infinity and where $C=\Pi_d$. Based on Proposition~\ref{prop:asso}, we may relax this rather restrictive assumption on $C$ considerably: we only require that certain covariances are zero and show that this condition is fulfilled if, for example, $C_I=\Pi_4$ for all $I\in\mathcal I_4(d)$. In fact, further relaxations will be given in the next subsection. 

More precisely, under $H_0$, the functions $g_I^\gamma$ from \eqref{eq:rhogi} simplify to
\begin{align*}
g_I^\beta(\bm u) &= 1-2 \cdot \bm 1( \mathrm{sgn}((u_\ell-\tfrac12)(u_m-\tfrac12))<0),
\\
g_I^\rho(\bm u) &= 12 (u_\ell-\tfrac12)(u_m-\tfrac12), \\
g_I^\tau(\bm u) &=  8 (u_\ell-\tfrac12)(u_m-\tfrac12),
\end{align*}
with $v_\gamma := \Var_C(g_I^\gamma(\bm U_1)) = \bm 1(\gamma\in\{\rho, \beta\})+\tfrac49\bm1(\gamma=\tau)$ for all $I \in \mathcal I_2(d)$.
If we additionally assume that  $C_I=\Pi_4$ for all $I\in\mathcal I_4(d)$, then a straightforward calculation shows that, for $I,J\in \mathcal I_2$,
\begin{align} \label{eq:covg}
r^\gamma_{I,J} = \Cov_C(g_{I}^\gamma(\bm U_i), g_{J}^\gamma (\bm U_i))= v_\gamma \bm 1(I=J). 
\end{align}
Let 
$
(Y_{d,I}^\gamma)_{I \in \mathcal I} \sim \mathcal N_{d(d-1)/2}(0, v_\gamma\bm I_{d(d-1)/2}),
$
where $\bm I_{d(d-1)/2}$ is the $d(d-1)/2$-dimensional identity matrix. Further, let
\[
Z_n^\gamma=\max_{I \in \mathcal I_2(d)} |Y_{d,I}^\gamma| = \max \Big\{ \max_{I \in \mathcal I_2(d)} Y_{d,I}^\gamma, \max_{I \in \mathcal I_2(d)} - Y_{d,I}^\gamma \Big\}.
\]
Theorem 2.2 in \cite{Che13} and Lemma 1 in \cite{Deo72} implies 
the following result.

\begin{corollary} \label{cor:gaussapp}
Suppose that the null hypothesis $H_0:(C_I=\Pi_2 \text{ for all } I \in \mathcal I_2(d))$ is met. If  \eqref{eq:covg} holds (which, for instance, is an immediate consequence of $C_I=\Pi_4$
for all $I\in\mathcal I_4(d)$) and if $d=d_n$ satisfies $\log d =o(n^{1/5})$, then, for $\gamma \in \{\rho, \tau, \beta\}$,
\begin{align} \label{eq:gaussapp}
\lim_{n\to\infty}\sup_{t \in \R} | \Prob(T_n^\gamma \le t) - \Prob(Z_n^\gamma \le t) | = 0.
\end{align}
Moreover, if $d=d_n$ grows to infinity, then $T_{n}^\gamma$ is asymptotically Gumbel-distributed: with $c_n = d_n(d_n-1)/2$, we have, for any $t\in\R$,
\begin{align} \label{eq:gumbelapp}
\lim_{n \to \infty}\Prob\Big[ \sqrt{2 \log c_n}\Big\{  v_\gamma^{-1/2} T_n^\gamma - \Big( \sqrt{2 \log c_n} - \frac{ \log(4\pi \log c_n)-4}{2 \sqrt{2\log c_n}} \Big) \Big\}  \le t \Big] = e^{-e^{-t}}.
\end{align}
\end{corollary}

\begin{proof}[Proof of Corollary~\ref{cor:gaussapp}]
The assumption on $C$ implies that the conditions of Theorem~\ref{theo:main} are met.
The assertion in \eqref{eq:gaussapp} is an immediate consequence of Proposition~\ref{prop:asso} and Theorem 2.1 in \cite{Che22}. The assertion in \eqref{eq:gumbelapp} then follows from Lemma~1 in \cite{Deo72}.
\end{proof}

\begin{remark}
The result in Corollary~\ref{cor:gaussapp} is weaker than the one in \cite{HanCheLiu17} in that we require a slightly stronger condition on the rate of $d=d_n$ (precisely, $\log d =o(n^{1/5})$ compared to $\log d =o(n^{1/3})$), but at the same time much stronger as we only require \eqref{eq:covg}. The latter in turn is, for instance, a straightforward consequence of the assumption that all four-dimensional margins of $C$ are independent; note that \cite{HanCheLiu17} require mutual independence $C=\Pi_d$.

In fact, it is possible to further relax the correlation condition \eqref{eq:covg} in Corollary~\ref{cor:gaussapp}. Indeed, in case that correlation condition is not met, the Gaussian approximation in \eqref{eq:gaussapp} still holds, but with $(Y_{d,I}^\gamma)_{I \in \mathcal I_2(d)}$ being multivariate normal with covariance matrix given by $(r^\gamma_{I,J})_{I,J \in \Ic_2(d)}$ \citep{Che22}. Moreover, under suitable assumptions controlling the strength of the correlations in that matrix, the maximum $\max_{I \in \Ic_2(d)} |Y_{d,I}^\gamma|$ still satisfies \eqref{eq:gumbelapp} as a consequence of Theorem 1 in \cite{Deo72}. An example is worked out below, but we refrain from formulating general sufficient conditions. Instead, in Section~\ref{sec:fer} below, we will provide alternative and powerful bootstrap approximations that allow for controlling type-I error rates under general assumptions controlling only the pairwise dependencies.
\end{remark}

\begin{example}
An example where \eqref{eq:gumbelapp} is met even though \eqref{eq:covg} is not satisfied is given by the following blockwise inductive model: let $V_1, V_2 \sim \mathrm{Uniform}([0,1])$ and $V_3 := (V_1+V_2) - \bm 1(V_1
+V_2>1)$. Then, $\bm V= (V_1, V_2, V_3)^\top$ has uniform margins and cdf $D\ne \Pi_3$ that satisfies $D_I=\Pi_2$ for all $|I|=2$. Some straightforward calculations imply that $r^\rho_{\{1,2\}, \{2,3\}} = r^\rho_{\{1,2\}, \{1,3\}}=2/5$ and $r^\rho_{\{1,3\}, \{2,3\}}=-2/5$.
Now, for $d$ an integer multiple of $3$, let $\bm U \in [0,1]^{d}$ have the same distribution as $d/3$ independent copies of $\bm V$. We then obtain that
\[
r^\rho_{I,J} = 
\begin{cases} 
1, &  I=J, \\ 
- \frac25, & \exists k \in \N_0: I,J \subset \{3k+1, 3k+2, 3k+3\}, I \cap J = \{3k+3\},  \\
\frac25, & \exists k \in \N_0: I,J \subset \{3k+1, 3k+2, 3k+3\}, I \cap J = \{3k+1\} \text{ or }  I \cap J = \{3k+2\},\\
0,& \text{else},
\end{cases}
\]
so \eqref{eq:covg} is violated.
For $d\in\N$, let 
$
(Y_{d,I}^\rho)_{I \in \mathcal I} \sim \mathcal N_{d(d-1)/2}(0, \bm \Sigma),
$
where $\bm \Sigma$ has entries $\bm \Sigma_{I,J} = r^\rho_{I,J}$. Note that $\bm \Sigma$ has exactly $2d$ off-diagonal elements that are non-zero, of which
$2d/3$ are equal to $-2/5$ and $4d/3$ are equal to $2/5$. 
To apply Theorem 1 in \cite{Deo72}, let $I_1, I_2, \dots$ denote an enumeration of all sets $\{i,j\}$ with $i,j \in \N$ and $i\ne j$. 
The enumeration should be in such a way that $I_1=\{1,2\}, I_2=\{1,3\}, I_3=\{2,3\}$ and:
\begin{compactitem}
\item For any $k\in \{4,7,10,\dots\}$, sets containing an index $\ell \ge k+3$ appear in the enumeration only after all sets $\{i,j\}$ with $i,j \le k+2$ have appeared (for instance, the sets $I_4, \dots, I_{15}$ are made up from all sets $\{i,j\}$ with $i\in \{1, \dots, 6\}$ and $j \in \{4,5,6\}$, and we then move on with the sets containing all $i\in \{1, \dots, 9\}$ and $j \in \{7,8,9\}$). 
\item For any $k\in \{4,7,10,\dots\}$, the first appearance of $k$ in the enumeration is at index $n=(k-1)(k-2)/2+1$ and it is of the form $I_n=\{k, k+1\}$ and then $I_{n+1}=\{k, k+2\}, I_{n+2}=\{k+1, k+2\}$ (for instance, $I_4=\{4,5\}, I_5=\{4,6\},I_6=\{5,6\}$ and $I_{16}=\{7,8\}, I_{17}=\{7,9\}, I_{18}=\{8,9\}$).
\end{compactitem}
Let $(\eta_i)_{i\in\N}$ denote a Gaussian sequence with $\Exp[\eta_i]=0$ and
$r_{ij}=\Cov(\eta_i, \eta_j) := r^\rho_{I_i, I_j}$ and note that, by construction, $r_{ij}=0$ for $|i-j|\ge 3$ and $|r_{ij}| \le 2/5$ for $|i-j|\le 2$. We may hence apply Theorem 1 in \cite{Deo72} to obtain that the maximum $\max_{i=1}^n |\eta_i|$ is asymptotically Gumbel with the same scaling sequences as if the $\eta_i$ were iid standard normal. This implies the assertion since $(\mathcal L(Z_n^\rho))_n$ is just a subsequence of $(\mathcal L(\max_{i=1}^n \eta_i))_n$ in the space of probability measures on the real line, where $\mathcal L(X)$ denotes the distribution of $X$.
\end{example}

\subsection{Controlling family-wise error rates in pairwise independence tests} \label{sec:fer}

Corollary~\ref{cor:gaussapp} allows to deduce asymptotic control of the type-I error rate of tests for the global null hypothesis $H_0:(C_I=\Pi_2 \text{ for all } I \in \mathcal I_2(d))$, based on the test statistics $T_n^\gamma=\sqrt n \max_{I \in \mathcal I_2(d)}|\hat \gamma_{n,I}|$ with $\gamma\in\{\rho, \tau,\gamma\}$, for all $H_0$-copulas that additionally satisfy $C_I=\Pi_4$ for all $I\in \mathcal I_4(d)$. While this is actually much weaker than what is required in related results from the literature \citep{HanCheLiu17},
it is still not satisfying. First of all, by design, tests based on $\sqrt{n}|\hat \gamma_{n,I}|$ should actually be regarded as tests for $H_{0,I}^\gamma: \gamma_I=0$, and control of the type-I error rate should be derived under $\gamma_I=0$, and not under $C_I=\Pi_2$, yet even $C=\Pi_d$. Next, in view of the multiplicity of the global testing problem, it would be desirable to control the  family-wise error rate, uniform over a large class of data-generating processes; see \eqref{eq:fer} below for a precise definition. 

More precisely, specializing to the case $\gamma=\rho$ for simplicity, let $H_I: \rho_I \le 0$ (extensions to testing $H_I:\rho_I=0$ are straightforward, but notationally more involved). We are looking for a procedure that simultaneously  tests $H_I$ against $H_I':\rho_I > 0$ with strong control of the family-wise error rate. This means that, for given confidence level $\alpha$, we want to control the probability of at least one false rejection by $\alpha+o(1)$, uniformly over a large class of data generating processes. 
Formally, fix some constant $K>0$ and an integer sequence $d=d_n$ as in Theorem~\ref{theo:main}, i.e., with $\log d =o(n^{1/3})$. Let $\Gamma=\Gamma^{(n)}(K)$ denote the set of all $d$-dimensional copulas $C^{(n)}$ for which Condition~\ref{cond:c2-new} is met with constant $K$ and $k=2$. Let $C_0=C_0^{(n)}\in \Gamma$ denote the true copula, and for $I \in \mathcal I=\mathcal I_2(d)$ let $\Gamma_I = \Gamma_I^{(n)} \subset \Gamma$ denote the subset of copulas for which $H_I$ is met.
For each $\mathcal J \subset \mathcal I$, let $\Gamma(\mathcal J)=\Gamma^{(n)}(\mathcal J)=(\bigcap_{I \in \mathcal J}  \Gamma_I) \cap (\bigcap_{I \notin \mathcal J} \Gamma_I^c)$, where $\Gamma_I^c=\Gamma \setminus \Gamma_I$, i.e., $\Gamma(\mathcal J)\subset \Gamma$ is the subset of copulas for which $H_I$ holds if and only if $I\in \mathcal J$. Asymptotic strong control of the family-wise error rate means
\begin{align}
\label{eq:fer}
\limsup_{n\to\infty}\sup_{\mathcal J \subset \mathcal I} \sup_{C_0 \in \Gamma^{(n)}(\mathcal J)} \Prob_{C_0}( \text{reject at least one hypothesis } H_I \text{ with } I \in \mathcal J) \le \alpha.
\end{align}

To ensure strong control of the family-wise error rate, we follow \cite{Che13} and apply the stepdown procedure from \cite{RomWol05} combined with the multiplier bootstrap. 
Write $x_{i,I} = g_I^\rho(\bm U_i)$ with $I=\{\ell,m\}$ and $g_I^\rho$ from \eqref{eq:rhogi}. Observable counterparts of $x_{i,I}$ are given by
\begin{align*}
\hat x_{i,I} 
&= 
12(1-\hat U_{i,\ell})(1-\hat U_{i,m}) - 36 \int \Pi_2 \diff \hat C_{n,I} + 12 \int_0^1  \hat C_{n,I}(\hat U_{i,\ell},z) + \hat C_{n,I}(z, \hat U_{i,m}) \diff z
\\&= 
12(1-\hat U_{i,\ell})(1-\hat U_{i,m}) \\
&\hspace{2.1cm} 
- \frac{12}n \sum_{j=1}^n \Big\{ 3 \hat U_{j,\ell} \hat U_{j,m} - \bm 1(\hat U_{j,\ell} \le \hat U_{i,\ell}) (1-\hat U_{j,m}) - \bm 1(\hat U_{j,m} \le \hat U_{i,m}) (1-\hat U_{j,\ell}) \Big\}.
\end{align*}
Let $e_1, \dots, e_n$ denote an iid sample of standard normal random variables that is independent of the observed data, and define $\hat T_{n,I}= n^{-1/2} \sum_{i=1}^n e_i \hat x_{i,I}$. Heuristically, if $H_I$ is met, the conditional distribution of $\hat T_{n,I}$ given $\bm X_1, \dots, \bm X_n$ should stochastically dominate the distribution of  $T_{n,I} := \sqrt n \hat \rho_{n,I}$ (and be close if $\rho_I=0$, i.e., on the boundary of $H_I$). For $\mathcal J \subset \mathcal I$, let $\hat c_{1-\alpha, \mathcal J}$ denote the $(1-\alpha)$-quantile of the conditional distribution of $\max_{I \in \mathcal J} \hat T_{n,I}$ given $\bm X_1, \dots, \bm X_n$ (which can be approximated to an arbitrary precision by repeated simulation). The stepdown procedure works as follows: in the first step, let $\mathcal J(1):= \mathcal I$. Reject all null hypotheses $H_I$ for which $T_{n,I}>\hat c_{1-\alpha, \mathcal I}$. If no null hypothesis is rejected, then stop. Otherwise, let $\mathcal J(2)$ denote the set of null hypotheses that were not rejected. On step $m \ge 2$, let $\mathcal J(m)$ denote the set of null hypotheses that were not rejected up to step $m$. Reject all null hypotheses $H_I$ for which $T_{n,I}>\hat c_{1-\alpha, \mathcal J(m)}$. If no null hypothesis is rejected, then stop. Otherwise, let $\mathcal J(m+1)$ denote the set of null hypotheses that were not rejected. Repeat until the algorithm stops, which eventually provides a decision on the acceptance or rejection of each $H_I$ with $I\in\mathcal I$. 

Under two additional weak assumptions, one on the class of data-generating processes and one on $d=d_n$, we obtain strong control of the family-wise error rate.  

\begin{theorem}[Strong control of the family-wise error rate for $\gamma=\rho$]
\label{theo:fwr}
Fix $c_1,K>0$. For $n\in\N$ let $\Gamma^{(n)}=\Gamma^{(n)}(c_1, K)$ denote the set of $d=d_n$-dimensional copulas $C=C^{(n)}$ for which Condition~\ref{cond:c2-new} is met with constant $K$ and with $k=2$, and for which $\Var_C(x_{1,I}^2) \ge c_1$  for all $I\in\mathcal I_2(d)$. If there exists a constant $c>0$ such that $(\log d)^7\le n^{1-c}$, then \eqref{eq:fer} is met.
\end{theorem}

\subsection{Max-type statistics based on the Moebius transform of the empirical copula process}
\label{sec:app3}

The problem of testing independence in random vectors of fixed dimension based on copula methods has attracted a lot of attention, see  \cite{Deh81, GenRem04, GenQueRem07, KojHol09, GenNesRemMur19}, among others. The methodology has recently been transferred to the high-dimensional regime via `sum-type' statistics in \cite{BucPak24}. Our main result in this paper allows to tackle respective `max-type' statistics; the relation of the current section to \cite{BucPak24} is hence comparable to the relation of \cite{HanCheLiu17} to \cite{LeuDrt18}, who consider certain `max-type' and `sum-type' tests for pairwise independence, respectively. 

For a function $G:[0,1]^d \to \R$ and $I\subset \{1, \dots, d\}$ with $|I|\ge 2$, consider the Moebius transform of $G$ with respect to $I$, defined as
\[
\Mc_I(G)(\bm u) =  \sum_{B \subset I} (-1)^{|I \setminus B|} G(\bm u^{B}) \prod_{j \in I \setminus B} G(u^{\{j\}}),
\]
where $\bm u \in [0,1]^d$ and where $\bm u^B \in [0,1]^d$ 
has $j$th component
$u_j^B =  u_j \bm 1(j\in B) + \bm 1(j\notin B)$. The tuple $(\Mc_I(G))_{I \subset \{1, \dots, d\}:|I|\ge 2}$ is called the Moebius decomposition of G. Note that $C=\Pi_d$ is equivalent to the fact that $\Mc_I(C)=0$ for all $I \subset \{1, \dots, d\}$ with $|I| \ge 2$ (\cite{GhoKulRem01}, Proposition 2.1).
Now $\Mc_I(C)=0$ is equivalent to $\int_{[0,1]^d} \Mc_I(C)^2 \diff \Pi_d=0$, and we 
may use 
\[
S_{n,I} = n \int_{[0,1]^d} \{ \Mc_I(\hat C_n)(\bm u) \}^2 \diff \Pi_d(\bm u)
\]
as an empirical version of the latter; see \cite{BucPak24}. Note that 
\begin{align*} 
S_{n,I} 
=
\frac1n \sum_{i_1, i_2=1}^n   \prod_{j \in I}  I_{i_1,i_2}^{(j)},
\quad
I_{i_1,i_2}^{(j)} = \frac{(n+1)(2n+1)}{6n^2} + \frac{R_{i_1j}(R_{i_1j}-1)}{2n^2} + \frac{R_{i_2j}(R_{i_2j}-1)}{2n^2}  - 
\frac{\max(R_{i_1j}, R_{i_2j})}{n},
\end{align*}
which allows for easy computation (note that the formula for $I_{i_1, i_2}^{(j)}$ is slightly different from the one in Formula (2.3) in \cite{BucPak24} due to a slightly different definition of the pseudo-observations -- the difference is asymptotically negligible though). We are interested in distributional approximation results for
$M_n(k)=\max_{I\in \mathcal I_k(d)} S_{n,I} $,
which can be considered as a test statistic for $H_k:(\Mc_I(C)=0$ for all $I\in \mathcal I_k(d))$. The  hypothesis is related to higher-order Lancaster interactions \citep{Lan69} as discussed in \cite{BucPak24}, and in the case $k=2$ corresponds to pairwise independence.

As in related papers \citep{HanCheLiu17, LeuDrt18, Drt20}, we only consider distributional approximations in the case $C=\Pi_d$, which we assume from now on. The main difficulty which prevents us from directly applying known results from the literature (notably Theorem~4.2 in \cite{Drt20} on degenerate U-statistics) are the ranks in $S_{n,I}$, and we proceed by sketching how Theorem~\ref{theo:main} helps to remove the ranks. First, consider a variant of the Moebius transform defined as 
\[
\widebar \Mc_I(G)(\bm u)  = \sum_{I \subset A} (-1)^{|I \setminus B|} G(\bm u^{B}) \prod_{j \in I \setminus B} u_j,
\]
and note that, unlike $\Mc_I$, $\widebar \Mc_I$ is linear in $G$.
We will show below that 
\begin{align} \label{eq:moeb0}
\max_{I \subset \{1, \dots, d\}: 2 \le |I| \le k} \|\Mc_I(\hat C_n)  - \widebar \Mc_I(\hat C_n)\|_\infty \le 2^k \frac1n.
\end{align}
By linearity of $\widebar \Mc_I$ and the fact that $\widebar \Mc_I(\Pi_d) =\Mc_I(\Pi_d)=0$, we obtain that 
\begin{align} \label{eq:moeb1}
\max_{I \subset \{1, \dots, d\}: 2 \le |I| \le  k} \|\sqrt n \Mc_I(\hat C_n)  - \widebar \Mc_I(\Cb_n)\|_\infty \le 2^k \frac1{\sqrt n} =O(n^{-1/2}).
\end{align}
Next, define $\beta_n(\bm u) = \sum_{j=1}^d ( \prod_{\ell \ne j} u_\ell) \alpha_{nj}(u_j)$ such that $\bar \Cb_n$ from \eqref{eq:barcn} can be written as $\bar \Cb_n=\alpha_n - \beta_n$. It has been shown in the proof of Proposition 2 in \cite{GenRem04} that $\widebar \Mc_I(\beta_n)=0$ for all $I$ with $|I| \ge 2$. Therefore, if $\log d =o(n^{1/3})$, it follows from Theorem~\ref{theo:main} and linearity of $\widebar \Mc_I$ that
\begin{align} \label{eq:moeb2}
\widebar \Mc_I(\Cb_n) 
&= \nonumber
\widebar \Mc_I(\bar \Cb_n) + \widebar \Mc_I(\Cb_n - \bar \Cb_n) 
\\&= 
\widebar \Mc_I(\alpha_n) + O_\as(n^{-1/4}(\log(nd))^{3/4})
= \widebar \Mc_I(\alpha_n) + o_\as(1),
\end{align}
where the error term is uniform in $I \subset \{1, \dots, d\}$ such that $2 \le |I| \le k$. Overall, from \eqref{eq:moeb1} and \eqref{eq:moeb2} and since the integration domain is finite, we obtain that
\[
\max_{2 \le |I| \le k} | S_{n,I} - \bar S_{n,I}| = o_{\as}(1),
\]
where
\[
\bar S_{n,I} := \int_{[0,1]^d} \widebar \Mc_I(\alpha_n)^2 \diff \Pi_d
=
\frac1n\sum_{i_1,i_2=1}^n h_I(\bm U_{i_1}, \bm U_{i_2})
\]
with
\[
h_I(\bm u, \bm v) = \prod_{j \in I } \Big\{ \frac{u_j^2}2 + \frac{v_j^2}2-\max(u_j, v_j) + \frac13\Big\}.
\]
We thus have removed the ranks and decomposing $\bar S_{n,I}=U_{n,I} + V_{n,I}$ where
\begin{align} \label{eq:univni}
U_{n,I} = \frac1{n}\sum_{1 \le i_1 \ne i_2 \le n} h_I(\bm U_{i_1}, \bm U_{i_2}), \qquad V_{n,I} = \frac1{n}\sum_{1 \le i \le n} h_I(\bm U_{i}, \bm U_{i}),
\end{align}
we may proceed by applying available distributional limit results for maxima of (degenerate) U-statistics of order 2 for independent observations; notably Theorem 4.2 in \cite{Drt20}. In view of the fact that such results are only available for $|I| =2$ to the best of our knowledge, we subsequently restrict attention to this case (which is arguably the most important case for statistical applications). We then obtain the following the result, which is proven in Section~\ref{sec:proofapp}.

\begin{proposition} \label{prop:moeb}
Let $C=\Pi_d$, and assume that $d=d_n$ grows to infinity such that $\log d =o(n^{1/8-\delta})$ for some $\delta>0$. Then
\begin{align} \label{eq:sni} 
\lim_{n\to\infty}\Prob\Big( \pi^4 \max_{I\in \mathcal I_2(d)}S_{n,I} - u_n \le y \Big) = \exp\Big\{ - \Big(\frac{\kappa^2}{8\pi}\Big)^{1/2} \exp\Big(-\frac{y}2\Big) \Big\}.
\end{align}
where
\[
u_n =  4 \log d_n - \log \log d_n -  \frac{\pi^4}{36}, \qquad
\kappa^2 = 
2 \prod_{n=2}^\infty \frac{\pi/n}{\sin(\pi/n)} \approx 6.086.
\]
\end{proposition}
Note that the limit distribution is of Gumbel-type, and that it is the same as for max-type tests involving Hoeffding's $D$, Blum–Kiefer–Rosenblatt's $R$ and Bergsma–Dassios–Yanagi\-moto's~$\tau^*$, see Examples 3.1-3.3 and Theorem 4.2 in \cite{Drt20}; it is only the scaling sequences on the right that are slightly different than for those examples.

\section{Proof of Theorem~\ref{theo:main}}
\label{sec:proofs}

Theorem~\ref{theo:main} will be a straightforward consequence of the Borel-Cantelli lemma and a non-asymptotic version of Stute's representation in probability. 

\begin{theorem}[Non-asymptotic Stute representation in probability] 
\label{theo:main-new}
Fix $k \in \N_{\ge 2}$ and $d\in\N_{\ge 2}$ with $k\le d$ and suppose that Condition~\ref{cond:c2-new} is met. Then there exist constants $L_1, L_2>0$ only depending on $k$ and $K$ such that, for any $\delta>0$ and $n \in \N$ with $e^{-n/32} \le \delta \le e^{-1}$, 
\[
\sup_{\bm u \in W_k} |\Cb_n(\bm u)- \bar \Cb_n(\bm u)| \le  
\frac{2k}{\sqrt n} + L_1 \sqrt{r \log\Big(\frac{L_2}{\delta r}\Big)} = \lambda_{n,k}(\delta) 
\]
with probability at least $1-(9d+d^k)\delta$,
where
\begin{align} \label{eq:r-def}
r := r_{n,\delta} :=  \sqrt{\frac1{2n} \log\Big(\frac1\delta\Big)}
\end{align}
(note that $(2n)^{-1/2} \le r\le 1/8$ by the assumption $e^{-n/32} \le \delta \le e^{-1}$).
\end{theorem}

\begin{proof}[Proof of Theorem~\ref{theo:main}]
Let $\delta_n=1/(dn)^k$, which implies that $(9d+d^k)\delta_n$ is summable. 
Observing that  $e^{-n/32} \le \delta_n \le e^{-1}$ for sufficiently large $n$, we may apply Theorem~\ref{theo:main-new} with $\delta=\delta_n$. The upper bound in that theorem then becomes
\begin{align*}
    \lambda_{n,k}(\delta_n) 
    &=
   \frac{2k}{\sqrt n}  + L_2  \Big[\frac{k\log(nd)}{n}\Big]^{1/4}
    \sqrt{\log\Big( \frac{L_1\sqrt 2 d^kn^{k+1/2}}{\sqrt{k\log(nd)}}\Big)} 
\\&\le
   \frac{2k}{\sqrt n}  + 
   L_2 \Big[\frac{k\log(nd)}{n}\Big]^{1/4}
    \sqrt{\log\Big( \frac{L_1 (dn)^{k+1/2}}{\sqrt{k\log(2)}}\Big)} 
\lesssim
       \frac{\{ \log(nd) \}^{3/4}}{n^{1/4}},
\end{align*}
where we used $d \ge 2$ at the first inequality. We conclude by the Borel-Cantelli lemma.
\end{proof}

\begin{proof}[Proof of Theorem~\ref{theo:main-new}]
Recall the definition of $G_n(\bm u) = \frac1n \sum_{i=1}^n \bm 1(\bm U_i \le \bm u)$ with margins $G_{n1}, \dots, G_{nd}$. For $u_j\in[0,1]$, define
$
G_{nj}^-(u_j) = \inf\{ v \in [0,1]: G_{nj}(v) \ge u_j\}
$
and let
\begin{align} \label{eq:cnt-new}
\tilde \Cb_n(\bm u) = \sqrt n \{ \tilde C_n(\bm u) - C(\bm u)\}, \quad \bm u \in [0,1]^d,
\end{align}
where
$\tilde C_n(\bm u) = G_n( G_{n1}^-(u_1), \dots, G_{nd}^-(u_d) )$.
By Lemma~\ref{lem:cnctn} is is sufficient to show that
\begin{align} \label{eq:stute-inverses}
\sup_{\bm u \in W_k} | \tilde \Cb_n(\bm u) - \bar \Cb_n(\bm u)|  \le 
\frac{k}{\sqrt n} + L_1 \sqrt{r \log\Big(\frac{L_2}{\delta r}\Big)} 
\end{align}
with probability at least $1-(9d+d^k)\delta$.
For that purpose, define
$\bm v_n(\bm u)=(v_{n,1}(u_1), \dots, v_{n,d}(u_d))$ where, for $j=1,\ldots,d$ and $u_j \in [0,1]$,
\[
v_{nj}(u_j) := \begin{cases} G_{nj}^-(u_j) & ,\quad u_j \ne 1 \\
1 &, \quad u_j=1.
\end{cases}
\]
Recall that $I_{\bm u}$ denote the set of indexes $j$ for which $u_j <1 $.
Note that $I_{\bm u} = I_{\bm v_n(\bm u)}$, i.e., $\bm u$ and $\bm v_n(\bm u)$ have entry 1 at the same coordinates.
In view of the fact that $G_{nj}^-(1)=\max_{i=1, \dots, n} U_{ij}$, we have
$
\tilde C_n(\bm u) = G_n(G_{n1}^-(u_1),  \dots, G_{nd}^-(u_d)) 
=
G_n(\bm v_n(\bm u)).
$
Therefore, we can write
\begin{align*}
\tilde \Cb_{n}(\bm u) 
&= \sqrt{n} \{ G_n\left( v_n(\bm  u)\right)  -C\left( v_n(\bm u ) \right) \}  +\sqrt{n}\{ C\left( v_n(\bm 
 u)\right) -C(\bm  u) \} 
\\
&=
\alpha_n(\bm v_n(\bm  u))  + \sqrt{n}\{  C\left(v_n(\bm u) \right)  -C(\bm u)\}.
\end{align*}
By Condition~\ref{cond:c1} and the mean value theorem, applied to the function 
$t \mapsto f(t) := C(\bm u + t(\bm v_n(\bm u) - \bm u))$, 
there exists a (random) $t^*=t_n^*(\bm u) \in (0,1)$ such that
\begin{align*}
\sqrt{n}\{  C\left(v_n(\bm u) \right)  -C(\bm u)\}
&=
\sum_{j \in I_{\bm u}} \dot C_j(\bm u + t^*(\bm v_n(\bm u)- \bm u)) \beta_{nj}(u_j),
\end{align*}
where $\beta_{nj}(u_j) = \sqrt n (G_{nj}^-(u_j) - u_j)$. 
As a consequence of the previous two displays and the fact that $\bar \Cb_n(\bm u) = \alpha_n(\bm u) - \sum_{j \in I_{\bm u}} \dot C_j(\bm u) \alpha_{nj}(u_j)$, for any $\bm u \in W_k$,
\begin{align} \label{eq:maindecomp-new}
\tilde \Cb_{n}(\bm u) - \bar \Cb_n(\bm u) 
&= 
S_{n1}(\bm u) + S_{n2}(\bm u) + S_{n3}(\bm u), \qquad \text{ where } \\
S_{n1}(\bm u) 
&=\nonumber
\alpha_n(\bm v_n(\bm  u)) - \alpha_n(\bm u), \\
S_{n2}(\bm u) 
&= \nonumber
\sum_{j\in I_{\bm u}} \dot C_j(\bm u + t^*(\bm v_n(\bm u)- \bm u)) \{ \beta_{nj}(u_j) + \alpha_{nj}(u_j) \}, \\
S_{n3}(\bm u) 
&= \nonumber
\sum_{j\in I_{\bm u}}  \big\{ \dot C_j(\bm u) - \dot C_j(\bm u + t^*(\bm v_n(\bm u)- \bm u))  \big\} \alpha_{nj}(u_j).
\end{align}
Lemmas \ref{lem:sn2-new}, \ref{lem:sn1-new} and \ref{lem:sn3-new} imply that
\begin{align*}
&\hspace{-1cm}\sup_{\bm u \in W_k} \big| S_{n1}(\bm u) + S_{n2}(\bm u) + S_{n3}(\bm u) \big|
\\&\le 
\lambda_{n,k}^{(1)}(\delta) + \lambda_{n,k}^{(2)} (\delta) + \lambda_{n,k}^{(3)} (\delta) 
\\ &=
\frac{k}{\sqrt n} + K_{1,1}  \sqrt{r \log\Big(\frac{K_{1,2}}{\delta r}\Big)} + K_2 k \sqrt{r \log\Big(\frac{160}{\delta r}\Big)}  + K_3 \sqrt{r \log\Big(\frac{80}{\delta r}\Big)} 
\end{align*}
with probability at least $1-(9d+d^k)\delta$. The assertion hence follows from \eqref{eq:stute-inverses} by setting $L_2=\max(K_{1,2}, 160)$ and $L_1 = K_{1,1}+K_2k + K_3$.
\end{proof}

\subsection{Reduction to generalized inverses}

Recall the definition of $\tilde \Cb_n$ from \eqref{eq:cnt-new}. The following result is well-known in the low-dimensional regime, and essentially the same proof applies in the high-dimensional regime. For completeness, it is provided in Appendix~\ref{sec:reduction-inverses}.

\begin{lemma} \label{lem:cnctn} 
Fix $k \in \N_{\ge 2}$.
Assume that, for each $j \in \{1, \dots, d\}$, the $j$th first-order partial derivative $\dot C_j$ exists and is continuous on $V_{j} \cap W_k$.
Then, with probability one,
\begin{align} \label{eq:cnctn}
\sup_{\bm u \in W_k} | \Cb_n(\bm u) - \tilde \Cb_n(\bm u)| \le \frac{k}{\sqrt n}.
\end{align}
\end{lemma}

\subsection{Bounding the error terms}

In this section we bound
$
\sup_{\bm u \in W_k} |S_{n\ell}(\bm u)|$
from \eqref{eq:maindecomp-new} for each $\ell=1,2,3$. 
The simplest term is $S_{n2}$, which we treat first.

\begin{lemma}[Bound on $S_{n2}$] \label{lem:sn2-new} 
Fix $k \in \N_{\ge 2}$ and $d\in\N_{\ge 2}$.
Assume that, for each $j \in \{1, \dots, d\}$, the $j$th first-order partial derivative $\dot C_j$ exists and is continuous on $V_{j} \cap W_k$. Then there exists a universal constant $K_2$ not depending on $k$ and $d$ such that, for any $\delta >0$ and $n \in \N$ with $e^{-n/2} \le \delta \le e^{-1}$,
\begin{align}\label{eq:hpsn1}
\sup_{\bm u \in W_k} |S_{n2}(\bm u)| 
\le 
\frac{k}{\sqrt n} + K_2 k  \sqrt{r \log\Big(\frac{160}{\delta r}\Big)} 
=: 
\lambda_{n,k}^{(2)}(\delta) 
\end{align}
with probability at least $1-3d\delta$ and with $r=r_{n,\delta}$ from \eqref{eq:r-def}.   
\end{lemma}

\begin{proof}
Since the partial derivatives of $C$ are bounded by 1, we have  
\[
|S_{n2}(\bm u)| \le 
T_{n2}(\bm u) :=
\sum_{j \in I_{\bm u}}
\big| \beta_{nj}(u_j) + \alpha_{nj}(u_j) \big| .
\]
Denote  the identity on $[0,1]$ by $\mathrm{id}$. As laid out in Section 15 (page 585) in \cite{ShoWel09}, we have
$
\alpha_{nj} + \beta_{nj}  = 
\alpha_{nj} - \alpha_{nj} \circ G_{nj}^{-} + \sqrt {n} (G_{nj} \circ G_{nj}^{-} - \mathrm{id})
$ 
and 
$
\| \sqrt {n} (G_{nj} \circ G_{nj}^{-} - \mathrm{id}) \|_\infty \le n^{-1/2},
$ 
where $\|\cdot\|_\infty$ denotes the supremum norm of a real-valued function. Hence,
\[
T_{n2}(\bm u)
\le 
\frac{|I_{\bm u}|}{\sqrt n} + \sum_{j \in I_{\bm u}} \omega_{\alpha_{nj}}(\|G_{nj}^- - \mathrm{id}\|_\infty),
\]
where $\omega_{f}(\eps)=\sup_{|u-v| \le \eps}|f(u)-f(v) |$, $\eps\ge 0$, denotes the modulus of continuity of $f:[0,1] \to \R$.
Using that $\|G_{nj}^- - \mathrm{id}\|_\infty = \|G_{nj} - \mathrm{id}\|_\infty$,
we obtain that
\[
\sup_{\bm u \in W_k} |S_{n2}(\bm u) |
\le
\frac{k}{\sqrt n} + k \max_{j=1}^d \omega_{\alpha_{nj}}(\|G_{nj} - \mathrm{id}\|_\infty).
\]
The classical DKW-inequality yields
\begin{align}
\label{eq:r-bound}
\max_{j=1}^d \|G_{nj} - \mathrm{id}\|_\infty \le r  
\end{align}
with probability at least $1-2d\delta$, with $r$ as in \eqref{eq:r-def}. On that event we hence obtain 
\begin{align}
\label{eq:proof-sn2-1}
\sup_{\bm u \in W_k} |S_{n2}(\bm u) |
\le
\frac{k}{\sqrt n} + k \max_{j=1}^d \omega_{\alpha_{nj}}(r).
\end{align}
By an inequality by Mason, Shorack and Wellner, formula (19) in Section 14.2 in \cite{ShoWel09} with their $\delta=1/2$,  we have, observing that $r \le 1/2$ by our assumption $\delta \ge e^{-n/2}$,
\begin{align}  
\Prob\Big( \max_{j=1}^d \omega_{\alpha_{nj}}(r) > \lambda\Big)
&\le 
d \times \Prob\Big( \omega_{\alpha_{n1}}(r) > \lambda\Big) 
\le \label{eq:bom1-new}
160 \frac{d}{r} \exp\Big( - \frac{\lambda^2}{32r} \psi\Big(\frac{\lambda}{n^{1/2}r}\Big)\Big),
\end{align}
where $\psi$ is the decreasing, continuous function on $[-1, \infty)$ defined by $\psi(-1)=2, \psi(0)=1$ and 
\begin{align} \label{eq:psi-new}
\psi(x) = 2x^{-2} \{ (1+x)\log(1+x) - x\}, \quad x \in (-1,0) \cup (0,\infty).
\end{align}
We will later choose $\lambda$ in such a way that $\lambda \le L \sqrt n r$, for some $L$ specified at the end of this proof. The fact that $\psi$ is decreasing then implies that $\psi(\lambda/(\sqrt n r)) \ge \psi(L)$. 
Hence, from \eqref{eq:bom1-new},
\begin{align*}
\Prob\Big( \max_{j=1}^d \omega_{\alpha_{nj}}(r) > \lambda\Big)
\le
160 \frac{d}{r} \exp\Big( - \frac{\lambda^2 \psi(L)}{32r} \Big).
\end{align*}
Choosing $\lambda=\sqrt{32r \log(160/(\delta r)) / \psi(L)}$, the upper bound in the previous display can be written as
$
160 dr^{-1} \exp( - \log (160\delta^{-1} r^{-1}))
 = d \delta.$
Together with \eqref{eq:proof-sn2-1}, this implies
\[
\sup_{\bm u \in W_k} |S_{n2}(\bm u) |
\le
\frac{k}{\sqrt n} + k \lambda
=
\frac{k}{\sqrt n} + \sqrt{\frac{32}{\psi(L)}} k r^{1/4} \sqrt{\log\Big( \frac{160 }{\delta r}\Big)} 
\]
with probability at least $1-3d \delta$. 
This yields the assertion, with $K_2=\sqrt{32/\{\psi(L)\}}$.

It remains to show the required inequality $\lambda \le L\sqrt n r$, which is equivalent to
\begin{align} \label{eq:sn2-inequality}
\log\Big( \frac{160 \sqrt{2n}}{\delta \sqrt{\log(1/\delta)}}\Big) \le \frac{L^2 \psi(L) \sqrt2 }{32} \sqrt{n \log(1/\delta)}.
\end{align}
The left-hand side can be bounded as follows:
\begin{align*}
\log\Big( \frac{160 \sqrt{2n}}{\delta \sqrt{\log(1/\delta)}}\Big) 
&=
\log(160 \sqrt 2) + \log(\sqrt n) + \log(1/\delta) - \frac12 \log(\log(1/\delta))
\\&\le
\log(160 \sqrt 2) + \log(\sqrt n) + \sqrt{ n/2 \cdot \log(1/\delta)} - 0
\end{align*}
where we used that $1 \le \log(1/\delta)\le n/2$ by assumption. Since $\log(160 \sqrt 2) + \log(\sqrt n) \le \log(160 \sqrt 2) \cdot  \sqrt {n}$ for all $n\in\N$, and again using $\log(1/\delta) \ge 1$, we obtain that
\begin{align}
\label{eq:log-bound}
\log\Big( \frac{160 \sqrt{2n}}{\delta \sqrt{\log(1/\delta)}}\Big) 
\le \{ \log(160 \sqrt 2) +2^{-1/2} \} \sqrt{n \log(1/\delta)}.
\end{align}
Hence, \eqref{eq:sn2-inequality} is met if we choose $L$ such that $L^2 \psi(L) \sqrt {2}/32 \ge \log(160 \sqrt 2) +2^{-1/2}$.
\end{proof}

\begin{lemma}[Bound on $S_{n1}$] \label{lem:sn1-new}
For any fixed $k \in \N_{\ge 2}$ and $d\in\N_{\ge 2}$ with $k\le d$ there exist constants $K_{1,1}, K_{1,2}>0$ only depending on $k$ such that, for any $\delta >0$ and $n \in \N$ with $e^{-n/2} \le \delta \le e^{-1}$,
\begin{align*}
\sup_{\bm u \in W_k} |S_{n1}(\bm u)|  \le 
K_{1,1}  \sqrt{r \log\Big(\frac{K_{1,2}}{\delta r}\Big)} =: \lambda_{n,k}^{(1)}(\delta)
\end{align*}
with probability at least $1-(2d+d^{k})\delta$, where $r=r_{n,\delta}$ is as in \eqref{eq:r-def}. 
\end{lemma}

\begin{proof}
Recall that, for $I\subset \{1, \dots, d\}$, $\alpha_{n,I}$ defined by 
$
\alpha_{n,I}(\bm w) := \alpha_n(\bm w^I)
$
denotes the $I$-margin of $\alpha_n$. Also, let  
$\bm v_{n,I}(\bm w) = (\bm v_n(\bm w^I))_I = (v_{nj}(w_j))_{j \in I} \in [0,1]^I$.
With these notations, we have, for any $\bm u \in [0,1]^d$, 
$
\alpha_{n}(\bm u) = \alpha_{n,I_{\bm u}}(\bm u_{I_{\bm u}}),
$
and in view of the fact that $I_{\bm u} = I_{\bm v_n(\bm u)}$, we also have
$
\alpha_{n}(\bm v_n(\bm u)) = \alpha_{n, I_{\bm u}}(\bm v_{n,I_{\bm u}}(\bm u_{I_{\bm u}})).
$
As a consequence,
\begin{align*}
\sup_{\bm u \in W_k} |S_{n1}(\bm u)| 
&= 
\sup_{\bm u \in W_k} |\alpha_{n,I_{\bm u}}(\bm u_{I_{\bm u}}) - \alpha_{n, I_{\bm u}}(\bm v_{n,I_{\bm u}}(\bm u_{I_{\bm u}})) |
\\ &=
\max_{I \subset \{1, \dots, d\}: |I|\le k} \sup_{\bm v \in [0,1]^I} | \alpha_{n,I}( \bm v) -  \alpha_{n,I}(\bm v_{n,I}(\bm v))|
\\ &\le
\max_{I \subset \{1, \dots, d\}: |I|\le k} \omega_{\alpha_{n,I}} \big( \max_{j=1}^d \|G_{nj} - \mathrm{id} \|_\infty \big),
\end{align*}
where, for a real-valued function $g$ on $[0,1]^{I}$ and $\eps\ge 0$, 
\[
\omega_g(\eps) = \sup\{ |g(\bm u) - g(\bm v)|: \bm u, \bm v \in [0,1]^I, |u_j - v_j| \le \eps \text{ for all } j \in I \}
\]
denotes the modulus of continuity of $g$. 
As in the proof of Lemma~\ref{lem:sn1-new}, see \eqref{eq:r-bound}, the classical DKW-inequality yields
$
\max_{j=1}^d \|G_{nj} - \mathrm{id}\|_\infty \le r = \sqrt{\log(1/\delta)/(2n)}
$
with probability at least $1-2d\delta$. On that event, we hence obtain
\begin{align*}
\sup_{\bm u \in W_k} |S_{n1}(\bm u)| 
&\le 
\max_{I \subset \{1, \dots, d\}: |I|\le k} \omega_{\alpha_{n,I}} (r).
\end{align*}

Next, by Proposition A.1 in \cite{Seg12}, there exist constants $M_1$ and $M_2$ only depending on $k$ such that
\begin{align}
\Prob \Big( \max_{I \subset \{1, \dots, d\}: |I|\le k} \omega_{\alpha_{n,I}} (r) > \lambda\Big)
&\le \nonumber
d^k\max_{I \subset \{1, \dots, d\}: |I|\le k} \Prob \Big(\omega_{\alpha_{n,I}} (r) > \lambda\Big)
\\&\le \label{eq:bom2-new}
\frac{M_1 d^k}{r} \exp\Big( - \frac{M_2 \lambda^2}{r} \psi\Big( \frac{\lambda}{n^{1/2}r}\Big)\Big),
\end{align}
where $\psi$ is defined in \eqref{eq:psi-new}. Note the similarity with \eqref{eq:bom1-new}; the remaining proof will hence be the same as the one of Lemma~\ref{lem:sn1-new}.

We will later choose $\lambda$ in such a way that $\lambda \le L \sqrt n r$, for some $L$ specified at the end of this proof. The fact that $\psi$ is decreasing then implies that $\psi(\lambda/(\sqrt n r)) \ge \psi(L)$. Hence, from \eqref{eq:bom2-new},
\begin{align*}
\Prob \Big( \max_{I \subset \{1, \dots, d\}: |I|\le k} \omega_{\alpha_{n,I}} (r) > \lambda\Big)
\le
M_1 \frac{d^k}{r} \exp\Big( - \frac{M_2 \lambda^2  \psi(L)}{r} \Big).
\end{align*}
Choosing $\lambda=\sqrt{r \log(M_1/(\delta r)) / (M_2 \psi(L))}$, the upper bound in the previous display can be written as $
M_1 d^kr^{-1} \exp( - \log ( M_1 \delta^{-1} r^{-1}))
= d^k \delta.
$
Overall, we have shown that 
\[
\sup_{\bm u \in W_k} |S_{n1}(\bm u)|  \le \lambda
=
\frac{1}{\sqrt{M_2 \psi(L)}} \sqrt{r \log\Big( \frac{M_1 }{\delta r}\Big)} 
\]
with probability at least $1-\delta (2d+d^{k})$. The assertion hence follows by setting $K_{1,1}=1/\sqrt{M_2 \psi(L) 2^{1/2}}$ and $K_{1,2}=M_1$.

It remains to show that $\lambda \le L \sqrt n r$, which can be guaranteed by choosing $L$ sufficiently large. Indeed, the same calculation that was used for proving \eqref{eq:sn2-inequality} shows that we must choose $L$ such that $L^2 \psi(L) M_2 \sqrt 2 \ge \log(M_1 \sqrt 2) + 2^{-1/2}$.
\end{proof}

\begin{lemma}[Bound on $S_{n3}$] \label{lem:sn3-new}
Fix $k \in \N_{\ge 2}$ and $d\in\N_{\ge 2}$ with $k\le d$ and suppose that Condition~\ref{cond:c2-new} is met. Then there exists a constant $K_{3}>0$ only depending on $k$ and $K$ such that, for any $\delta>0$ and $n \in \N$ with $e^{-n/32} \le \delta \le e^{-1}$, we have
\[
\sup_{\bm u \in W_k} |S_{n3}(\bm u)| \le  
K_3 \sqrt{r \log\Big(\frac{80}{\delta r}\Big)} =: \lambda_{n,k}^{(3)}(\delta)
\]
with probability at least $1-4d\delta$ and with $r=r_{n,\delta}$ from \eqref{eq:r-def}. 
\end{lemma}

\begin{proof}
Recall
$
S_{n3}(\bm u) 
= 
\sum_{j \in I_{\bm u}} D_{n,j}(\bm u) 
$
with 
\[
D_{n,j}(\bm u) = \big\{ \dot C_j(\bm u) - \dot C_j(\bm u + t^*(\bm v_n(\bm u)- \bm u))  \big\} \alpha_{nj}(u_j).
\]
Define $J_r=[0, 2r) \cup (1-2r ,1]$. We may then write $S_{n3}(\bm u) = S_{n3}^b(\bm u) + S_{n3}^c(\bm u)$, where
\[
S_{n3}^b(\bm u) = \sum_{j \in I_{\bm u}} D_{n,j}(\bm u) \bm 1 (u_j \in J_r), \qquad
S_{n3}^c(\bm u) = \sum_{j \in I_{\bm u}} D_{n,j}(\bm u) \bm 1 (u_j \notin J_r),
\]
the upper indexes $b$ and $c$ standing for `boundary' and `center', respectively.

We start by bounding the boundary part. For that purpose, let $\bm u \in W_k$ be arbitrary and let $j \in I_{\bm u}$. Then, since $\dot C_j \in [0,1]$,
\begin{align*}
|D_{nj}(\bm u)| \bm 1(u_j \in J_r)  \le |\alpha_{n,j}(u_j)| \bm 1(u_j \in J_r)
&\le \sup_{u_j \in J_r} |\alpha_{nj}(u_j)| 
= \sup_{u_j \in J_r} |\alpha_{nj}(u_j) - \alpha_{nj}(0)| 
\le 
\omega_{\alpha_{nj}}(2r).
\end{align*}
As a consequence, 
\[
\sup_{\bm u \in W_k} |S_{n3}^b(\bm u)| \le k \max_{j=1}^{d} \omega_{\alpha_{nj}}(2r).
\]

Hence, as in \eqref{eq:bom1-new}, for $\lambda>0$ specified below, an application of formula (19) in Section 14.2 in \cite{ShoWel09} implies
\begin{align*}
\Prob\Big( \max_{j=1}^{d} \omega_{\alpha_{nj}}(2r) > \lambda \Big)
&\le 
d \times \Prob\Big( \omega_{\alpha_{n1}}(2r) > \lambda  \Big)
\le
80 \frac{d}{r} \exp\Big( - \frac{\lambda^2}{64r} \psi\Big(\frac{\lambda}{2n^{1/2}r}\Big)\Big).
\end{align*}
We will later choose $\lambda$ in such a way that $\lambda \le 2L \sqrt n r$, for some $L$ specified at the end of this proof. The fact that $\psi$ is decreasing then implies that $\psi(\lambda/(2\sqrt n r)) \ge \psi(L)$.
Hence, 
\begin{align*}
\Prob\Big( \max_{j=1}^d \omega_{\alpha_{nj}}(2r) > \lambda\Big)
\le
80 \frac{d}{r} \exp\Big( - \frac{\lambda^2 \psi(L)}{64r} \Big).
\end{align*}
Choosing
$
\lambda = \sqrt{64 r \log(80/(\delta r)) / \psi(L)}
$
the upper bound in the previous display can be written as
$
80 dr^{-1}\exp( - \log( 80 \delta^{-1} r^{-1})) = d \delta.
$
Overall, 
\begin{align}
\label{eq:sn3b-bound}
\sup_{\bm u \in W_k} |S_{n3}^b(\bm u)| \le k \lambda = 8k\psi(L)^{-1/2} \sqrt{r \log\Big(\frac{80}{\delta r}\Big)}
\end{align}
with probability at least $1-d\delta$, provided we can choose $\lambda$ sufficiently large such that $\lambda\le 2L\sqrt n r$. This in turn follows exactly by the same arguments that lead to \eqref{eq:sn2-inequality}.

It remains to consider $\sup_{\bm u \in W_k} |S_{n3}^c(\bm u)| $. Let $\bm u \in W_k$ and $j\in I_{\bm u}$. Then, by Lemma~4.3 in \cite{Seg12} and with $K$ from Condition~\ref{cond:c2-new}, 
\begin{align*}
|D_{nj}(\bm u)|  \bm 1(u_j \notin J_r)  
&\leq  
K  \max \Big( \frac1{u_j(1-u_j)}, \frac1{G^{-}_{nj}(u_j)(1-G^{-}_{nj}(u_j))} \Big)  \Big(\sum_{\ell \in I_{\bm u}} |G_{n\ell}^-(u_\ell) - u_\ell|  \Big) | \alpha_{nj}(u_j)| \bm 1(u_j \notin J_r) 
\nonumber \\&\le 
Kk \times C_{n1} \times C_{n2},
\end{align*}
where
\begin{align*}
C_{n1} &= \max_{\ell=1}^d \| G_{n\ell}^- - \mathrm{id} \|_\infty 
\\ 
C_{n2} &= \max_{j =1}^d \sup_{u_j \notin J_r} \max \Big( \frac1{u_j(1-u_j)}, \frac1{G^{-}_{nj}(u_j)(1-G^{-}_{nj}(u_j))} \Big) | \alpha_{nj}(u_j)| .
\end{align*}
As a consequence, since $|I_{\bm u}| \le k$,
\begin{align} \label{eq:sn3c-1-new}
\sup_{\bm u \in W_k} |S_{n3}^c(\bm u)| \le K k^2 \times C_{n1} \times C_{n2}.
\end{align}
We proceed by bounding $C_{n1}$ and $C_{n2}$. First, by the classical DKW-inequality from \eqref{eq:r-bound} and since $\| G_{n\ell}^- - \mathrm{id} \|_\infty=\| G_{n\ell} - \mathrm{id} \|_\infty$, we have
\begin{align} \label{eq:cn1-new}
C_{n1} \le r 
\end{align}
with probability at least $1-2d\delta$. Subsequently, we work on this event.
Regarding $C_{n2}$, we then have, for $u_j \ge 2r$,
\begin{align*}
G_{nj}^-(u_j) 
= 
u_j \Big( 1+\frac{G_{nj}^-(u_j) - u_j}{u_j}\Big) 
 &\ge 
u_j \Big( 1 - \frac{\max_{j=1}^d \| G_{nj}^- - \mathrm{id} \|_\infty}{2r} \Big)
=u_j\Big(1-\frac{C_{n1}}{2r}\Big)
\ge \frac{u_j}{2}
\end{align*}
where we used \eqref{eq:cn1-new}. By the same argument we also have
$
1-G_{nj}^-(u_j) 
\ge 
(1-u_j)/2
$
for $u_j\ge 1-2r$.
As a consequence, 
\[
C_{n2} \le 4 \max_{j =1}^d \sup_{u_j \notin J_r}  \frac{| \alpha_{nj}(u_j)|}{u_j(1-u_j)}
\le 4 r^{-1/2} \max_{j =1}^d \sup_{u_j \notin J_r} \frac{| \alpha_{nj}(u_j)|}{(u_j(1-u_j))^{1/2}}  =: 4 r^{-1/2} D_n.
\]
Together with \eqref{eq:sn3c-1-new} and \eqref{eq:cn1-new}, we obtain that
\[
\sup_{\bm u \in W_k} |S_{n3}^c(\bm u)| \le 4  K k^2 r r^{-1/2} D_n = 4 K k^2 \sqrt r D_n.
\]
Note that 
\begin{align*}
\sup_{u \notin J_r} \frac{| \alpha_{n1}(u)|}{(u(1-u))^{1/2}} 
&= 
\max\Big\{ \sup_{u \in [2r, \frac12]} \frac{| \alpha_{n1}(u)|}{(u(1-u))^{1/2}}, 
\sup_{u \in [\frac12, 1-2r]} \frac{| \alpha_{nj}(u)|}{(u(1-u))^{1/2}} \Big\}
\\&\le
2\max\Big\{ \sup_{u \in [2r, \frac12]} \frac{| \alpha_{n1}(u)|}{u^{1/2}}, 
\sup_{u \in [\frac12, 1-2r]} \frac{| \alpha_{n1}(u)|}{(1-u)^{1/2}} \Big\},
\end{align*}
and that the two suprema inside the maximum on the right-hand side have the same distribution. Overall, 
by the union bound, for $\lambda>0$ specified below
\begin{align}
\Prob\Big(\sup_{\bm u \in W_k} |S_{n3}^c(\bm u)| > \lambda \Big)
&\le
2 d \times \Prob \Big(  \sup_{u \in [2r, \frac12]} \frac{| \alpha_{n1}(u)|}{u^{1/2}} > \frac{\lambda}{8Kk^2 \sqrt r} \Big).
\label{eq:sn3-c2-new}
\end{align}
We will next apply Corollary 11.2.1 in \cite{ShoWel09}, which states that, for $0 \le \eta \le 1/4$ and $\eps>0$,
\begin{align} \label{eq:sho1-new}
\Prob\Big( \sup_{\eta \le u \le 1/2} \frac{\alpha_{n1}^\pm(u)}{\sqrt u} \ge \eps \Big)
\le 6 \log\Big(\frac1{2\eta} \Big) \exp\Big( - \frac18 \gamma_{\pm} \eps^2\Big),
\end{align}
where  $a^+=\max(a,0)$ and $a^-=\max(-a,0)$ for $a\in\R$ and where $\gamma_-=1$ and 
\begin{align*}
\gamma_+= 
\begin{cases} \frac12 & \text{ if } \eps \le \frac32 (n \eta)^{1/2}, \\
\tfrac34 \frac{(n\eta)^{1/2}}{\eps} & \text{ if } \eps \ge \frac32 (n \eta)^{1/2}.
\end{cases}
\end{align*}
Suppose we have that $\gamma_+ \ge 1/c$ for some $c \ge 2$. Then $\gamma_{\pm} \ge 1/c$, and hence, since $|a|=a^++a^-$, Equation~\eqref{eq:sho1-new} implies
\begin{align*}
\Prob\Big( \sup_{\eta \le u \le 1/2} \frac{|\alpha_{n1}(u)|}{\sqrt u} \ge \eps \Big)
&\le 12 \log\Big(\frac1{2\eta} \Big) \exp\Big( - \frac1{32c} \eps^2\Big).
\end{align*}
We will later show that $\gamma_+ \ge 1/c$ is met if we choose $\eta=2r$ (note that $\eta \le 1/4$ by assumption) and $\eps=\lambda/(8Kk^2\sqrt r)$ as required in \eqref{eq:sn3-c2-new}, with $\lambda$ specified below. We then obtain, from \eqref{eq:sn3-c2-new},
\begin{align*}
\Prob\Big(\sup_{\bm u \in W_k} |S_{n3}^c(\bm u)| > \lambda \Big)
&\le
24 d \log\Big(\frac{1}{4r} \Big) \exp\Big( - \frac{\lambda^2}{c 2^{11}K^2k^4r} \Big). 
\end{align*}
This is equal to $d  \delta$ if we choose
\[
\lambda=\sqrt{c2^{11}K^2 k^4 r  \log\Big( \frac{24 \log (1/(4r)) }{\delta }\Big)}.
\]
Overall, we have shown that 
\begin{align}
\label{eq:sn3c-bound}
\sup_{\bm u \in W_k} |S_{n3}^c(\bm u)|
\le 
\sqrt{c}2^{11/2}K k^2 \sqrt{ r  \log\Big( \frac{24 \log (1/(4r)) }{\delta }\Big)}
\end{align}
with probability at least $1-3d\delta$. Together with \eqref{eq:sn3b-bound}, this implies
\[
\sup_{\bm u \in W_k} |S_{n3}(\bm u)| \le 
8k\psi(L)^{-1/2} \sqrt{r \log\Big(\frac{80}{\delta r}\Big)}
+
\sqrt{c}2^{11/2} K k^2 \sqrt{ r  \log\Big( \frac{24 \log (1/(4r)) }{\delta }\Big)}
\]
with probability at least $1-4d\delta$. Since $\log(x) \le x/e$ for $x\ge 1$, the second square root in the previous display is certainly smaller than the first one. Overall, we hence obtain
\[
\sup_{\bm u \in W_k} |S_{n3}(\bm u)| \le 
\big\{ 8k\psi(L)^{-1/2}+ \sqrt{c}2^{11/2} K k^2\big\} \sqrt{r \log\Big(\frac{80}{\delta r}\Big)}
\]
with probability at least $1-4d\delta$, which yields the assertion by defining $K_3$ as the term in curly brackets.

It remains to show that we can choose $c$ large enough so that $\gamma_+ \ge 1/c$, which, for the choices of $\eps, \eta$ and $\lambda$ from above, means that $\tfrac34 (n\eta)^{1/2}/\eps \ge 1/c$ must be satisfied. For proving this, we start by observing that
$\log(1/(4r)) \le 1/(4er)$, which implies
\begin{align*}
\log\Big( \frac{24 \log(1/(4r))} \delta \Big)
\le
\log\Big( \frac{6}e \frac1{\delta r} \Big)
&=
\log\Big( \frac{6 \sqrt 2}e \frac{\sqrt{n}}{\delta \sqrt{\log(1/\delta)}} \Big)
\\&\le 
\{\log(6\sqrt2 /e) +1/\sqrt{32} \}  \sqrt{n\log(1/\delta)},
\end{align*}
where the last inequality follows from the same calculation that lead to \eqref{eq:log-bound}. Thus,
\begin{align*}
\frac34\frac{\sqrt{n\eta}}\eps 
= 
\frac{3}{4} \frac{\sqrt{2nr}}{\lambda} 8 K k^2 \sqrt r    
& =
6\sqrt 2 K k^2    \frac{\sqrt{n}r} {\sqrt{c2^{11}K^2k^4r \log(24\log(1/(4r))/\delta)}  }
\\&\ge
\frac{6}{32 \sqrt{c}} \frac{\sqrt{nr}}{\sqrt{\{\log(6\sqrt2 /e) +1/\sqrt{32} \}  \sqrt{n\log(1/\delta)}}}
\\&=
\frac{3}{16 \sqrt{c} \sqrt{\{\log(6\sqrt2 /e) +1/\sqrt{32} \} \sqrt 2 } }.
\end{align*}
This expression is indeed larger than $1/c$ for sufficiently large $c$.
\end{proof}

\subsection{An almost sure bound on the maximum of empirical process marginals}

\begin{proposition}\label{prop:alphauni-all}
For $n\in \N$, let $\bm U_{1}^{(n)}, \dots, \bm U_{n}^{(n)}$ be independent $d=d_n$-dimensional random vectors with standard uniform margins, defined on a common probability space independent of $n$.
Fix $k\in\N$ and assume $d=d_n \ge k$ for all $n$. For $n\in\N$ let 
$\alpha_{n}(\bm u)=n^{-1}\sum_{i=1}^n \{ \bm 1(\bm U_{i}^{(n)} \le \bm u) - \Prob(\bm U_{i}^{(n)} \le \bm u) \}$ with $\bm u \in [0,1]^d$ denote the standard empirical process with $k$-dimensional margins
$\alpha_{n,I}(\bm u)=n^{-1}\sum_{i=1}^n \{ \bm 1(\bm U_{i,I}^{(n)} \le \bm u) - \Prob(\bm U_{i,I}^{(n)} \le \bm u) \}$ with $\bm u \in [0,1]^k$ for $I\in \mathcal I_k(d)$. 
Then
\begin{align*}
\limsup_{n \to \infty} \frac{\sup_{\bm u \in W_k} |\alpha_n(\bm u) |} {\sqrt{\log(nd^k)}}
=
\limsup_{n \to \infty} \frac{\max_{I \in \mathcal I_k(d)} \| \alpha_{n,I} \|_\infty} {\sqrt{\log(nd^k)}}
\leq \frac1{\sqrt{2}} \bm 1(k=1)+\bm 1(k\ge 2)
\end{align*}
almost surely.
\end{proposition}

\begin{proof}
We start by considering the case $k\ge 2$. Write  $A_n:=\max_{I\in \mathcal I_k(d)} \|\alpha_{n,I} \|_\infty$ and $a_n = \sqrt{\log (nd^k) }$.
For $\eps>0$, let $\lambda_n =\lambda_n(\eps)=(1+ \eps)a_n$  and note that it is sufficient to show that, for each $\eps>0$,
\begin{align}\label{eq:anbc2}
\Prob\Big( A_n > \lambda_n  \quad \text{ i.o.}\Big) =
\Prob\Big( \frac{A_n}{a_n}  > 1 + \eps \quad \text{ i.o.}\Big) = 0
\end{align}
where `i.o.' stands for `infinitely often'. Indeed, we then have $\Prob(\limsup_{n\to\infty} {A_n}/{a_n}>1) = \Prob(\exists \eps>0: A_n/a_n>1+\eps \text{ i.o}) = \lim_{\eps \downarrow 0}\Prob(A_n/a_n>1+\eps \text{ i.o.})=0$ by continuity of measures from below. Now, by the union bound and the multivariate Dvoretzky–Kiefer–Wolfowitz inequality, see Theorem 4.1 in \cite{Naa21}, we obtain
\begin{align*}
\Prob\Big(A_n > \lambda_n\Big) 
\le
\sum_{I \in \mathcal I_k(d)} \Prob\Big( \|\alpha_{n,I} \|_\infty > \lambda_n\Big)
&\le 
d^k \max_{I \in \mathcal I_k(d)} \Prob\Big( \| \alpha_{n,I} \|_\infty > (1+ \eps) \sqrt{ \log (nd^k) } \Big) \\
&\le 
d^k \times k(n+1) \exp\Big( - 2 (1+\eps)^2  \log (nd^k) \Big )
\\&= 
k (n+1) d^k (nd^k)^{-2 (1+\eps)^2} \lesssim n^{-1-4\eps-2\eps^2}. 
\end{align*}
We obtain that the sequence $\Prob(A_n > \lambda_n)$ is summable, which yields  \eqref{eq:anbc2} by the Borel-Cantelli lemma.

In the case $k=1$, the same proof applies, but with $a_n = \sqrt{\log(nd)/2}$ and with the multivariate DKW inequality from \cite{Naa21} replaced by the classical DKW inequality. \end{proof}

\begin{corollary} \label{corr:supcn}
Fix $k\in\N_{\ge 2}$ and assume that $d=d_n \ge k$ for all $n$ and that, for each $j \in \{1, \dots, d\}$, the $j$th first-order partial derivative $\dot C_j$ exists on $V_j \cap W_k$ (such that $\bar \Cb_n$ is well-defined on $W_k$). Then, 
\[
\sup_{\bm u \in W_k} |\bar \Cb_n(\bm u)| = O_\as((\log (nd))^{1/2}). 
\]
Moreover, under the conditions of Theorem~\ref{theo:main}, we have
\[
\sup_{\bm u \in W_k} |\Cb_n(\bm u)| = O_\as((\log (nd))^{1/2}). 
\]
\end{corollary}

\begin{proof} 
The first assertion is an immediate consequence of Proposition~\ref{prop:alphauni-all}, the fact that $|I_{\bm u}|\le k$ for any $\bm u \in W_k$, and boundedness of the first-order partial derivatives of $C$, i.e., $|\dot C_j| \le 1$ for all $j$. The second assertion then follows from an application of Theorem~\ref{theo:main}.
\end{proof}

\section{Proofs for Section~\ref{sec:app}} 
\label{sec:proofapp}

\begin{proof}[Proof of Proposition~\ref{prop:asso}]
Subsequently, we write $s_n=n^{-1/4}(\log(nd))^{3/4}$. We start by observing that $\int D \diff \hat C_{n,I} = \int \hat C_{n,I} \diff D + \frac1n$ for all bivariate copulas $D$ and all $I\subset \{1, \dots, d\}$ with $|I|=2$. Indeed, let $(U,V) \sim D, (U', V') \sim \hat C_{n,I}$ be independent. Then $\Prob(U' \le U, V' \le V)=\int \hat C_{n,I} \diff D$ and likewise $\Prob(U\le U', V \le V') = \int D \diff  \hat C_{n,I}$.
Moreover, by the inclusion-exclusion principle, 
$
\Prob(U' \le U, V' \le V) 
=1- \Prob(U' \ge U) - \Prob(V' \ge V) + \Prob(U' \ge U, V' \ge V).
$
Since $U'$ is uniformly distributed on $\{1/n, \dots, n/n\}$ and independent of $U\sim \mathrm{Uniform}([0,1)])$, we have
$
\Prob(V \le V') = \Prob(U \le U') = \int_0^1 \int_0^{u'} dF_U(u) \, dF_{U'}(u_n) = \int_0^1 u' \, dF_{U'} (u') = \sum_{i=1}^n \frac{i}{n}\frac{1}{n} = \frac{n+1}{2n},
$
with $F_X$ the cdf of a random variable $X$. As a consequence, 
$
\Prob(U' \le U, V' \le V) 
= 1-2\frac{n+1}{2n} + \Prob(U' \ge U, V' \ge V) = - n^{-1} + \Prob(U' \ge U, V' \ge V),
$
which yields the claimed identity.

Using the identity $\int D \diff \hat C_{n,I} = \int \hat C_{n,I} \diff D + \frac1n$, some tedious but straightforward calculations imply that  $\hat \rho_{n,I} =\tilde \rho_{n,I} + r_{n,I}^\rho$ and $\hat \tau_{n,I}=\tilde \tau_{n,I} + r_{n,I}^\tau$, where
\begin{align*}
\tilde \rho_{n,I} =
12 \int \hat C_{n,I} \diff \Pi_2 - 3, \qquad 
\tilde \tau_{n,I} = 4 \int \hat C_{n,I} \diff \hat C_{n,I} - 1,
\end{align*}
and where $|r_{n,I}^\rho| \le 6/(n-1)$ and $|r_{n,I}^\tau| \le 4/(n-1)$; details are provided in the supplement. 
As a consequence, it is sufficient to prove \eqref{eq:association} with $\hat \gamma_{n,I}$ replaced by $\tilde \gamma_{n,I}$ for $\gamma\in\{\rho, \tau\}$. 

We start with $\gamma=\beta$. In that case, since $S_{n,I}^\beta = 4 \bar \Cb_{n,I}(\tfrac12,\tfrac12)$, whence 
$\max_{I\in \mathcal I_2(d)} |\sqrt n(\hat \beta_{n,I}-\beta_I) - S_{n,I}^\beta|= 
\max_{I\in \mathcal I_2(d)} |4\Cb_{n,I}(\tfrac12,\tfrac12)-4\bar \Cb_{n,I}(\tfrac12,\tfrac12)| = 
 O_\as(s_n)$ by Theorem~\ref{theo:main}.

Next, consider $\gamma=\rho$. In that case, a straightforward calculation shows that $12 \int \bar \Cb_{n,I} \diff \Pi_2=S_{n,I}^\rho$. Indeed, we may assume that $d=2$ and $I=\{1,2\}$, and then write $(U_i,V_i)$ instead of $\bm U_{i}$. First, since $\int \bm 1(U_i \le u, V_i \le v) \diff \Pi_2(u,v) = (1-U_i)(1-V_i)$, we have $\int \alpha_{n}(u,v) \diff \Pi_2(u,v) = n^{-1/2} \sum_{i=1}^n (1-U_i)(1-V_i) - \int C \diff \Pi_2$. Next, $\int \dot C_1(u, v) \alpha_n(u,1) \diff \Pi_2(u,v) = - n^{-1/2} \sum_{i=1}^n  \int C(U_i, v) \diff v - \int C \diff \Pi_2$, since $\int C(1,v)\diff v = 1-\int \Prob(V>v)\diff v = 1-\Exp(V)=1/2$ and,
\begin{align*}
&\int \int \dot C_1(u,v) \bm 1(U_i \le u) \diff u \diff v= \int C(1,v)-C(U_i, v) \diff v = \frac12 - \int C(U_i, v) \diff v \\
&\int \int \dot C_1(u,v) u \diff u \diff v= \int C(1,v)\diff v - \int C(u,v) \diff u \diff v = \frac 12 - \int C \diff \Pi_2,
\end{align*}
where we used partial integration in the second line: $\int_0^1 \dot C_1(u,v) u \diff u = C(1,v)\cdot 1 - \int_0^1 C(u,v) \diff u$. Assembling terms yields the claimed formula.
As a consequence of that formula, we have
$\max_{I\in \mathcal I_2(d)} |\sqrt n(\tilde \rho_{n,I}-\rho_I) - S_{n,I}^\rho|=\max_{I\in \mathcal I_2(d)} |12 \int (\Cb_{n,I} - \bar \Cb_{n,I}) \diff \Pi_2| = O_\as(s_n)$ by Theorem~\ref{theo:main} and boundedness of the integration domain.

Finally, consider $\gamma=\tau$. Then, using once again the identity $\int C_I \diff \hat C_{n,I} = \int \hat C_{n,I} \diff C_I + \frac1n$ from the beginning of the proof,
\begin{align*}
\frac14\sqrt n (\tilde \tau_{n,I} - \tau_I) &= \sqrt n \Big\{ \int \hat C_{n,I} \diff \hat C_{n,I}-\int C_I \diff C_I \Big\}
\\&= 
\sqrt n \Big\{\int (\hat C_{n,I} - C_I) \diff\hat C_{n,I} + \int C_I \diff \hat C_{n,I} - \int C_I \diff C_I \Big\} \\
&=
\sqrt n \Big\{\int (\hat C_{n,I} - C_I) \diff \hat C_{n,I} + \int (\hat C_{n,I} - C_I) \diff C_I + \frac1n \Big\}
\\&=
2 \int \Cb_{n,I} \diff C_{I}  + R_{n,I},
\end{align*}
where 
$R_{n,I} =  \frac1{\sqrt n}  + \sqrt n \int (\hat C_{n,I} - C_I) \diff (\hat C_{n,I}-C_{I} )
=\frac1{\sqrt n}+
\int \Cb_{n,I} \diff (\hat C_{n,I}-C_{I} )$.
Below, we will show 
\begin{align} \label{eq:rntau}
\max_{I \in \mathcal I_2(d)}|R_{n,I}|=O_\as(s_n).
\end{align}
The assertion of the proposition then follows from Theorem~\ref{theo:main} and the fact that $2\int \bar \Cb_{n,I} \diff C_I = S_{n,I}^\tau$. For proving the latter, we may assume that $d=2$ and $I=\{1,2\}$, and then write $(U_i,V_i)$ instead of $\bm U_{i}$. By definition of the Stieltjes integral,
\begin{align*}
\int \bm 1(U_i \le u, V_i \le v) \diff C(u,v)
&=
\Exp_{(X,Y) \sim C}[ \bm 1(U_i \le X, V_i \le Y)]
\\&= 
\Exp_{(X,Y) \sim C}[ 1 - \bm 1(U_i > X) - \bm 1(V_i > Y) + \bm 1(U_i >X, V_i > Y)]
\\&=
1-U_i-V_i+C(U_i,V_i).
\end{align*}
As a consequence, $
\int \alpha_{n}(u,v) \diff C(u,v) = n^{-1/2} \sum_{i=1}^n \{1-U_i-V_i+C(U_i,V_i) - (\tau+1)/4 \}$, since $\int C \diff C=(\tau+1)/4$.
Next, note that 
\begin{align} \label{eq:dd1}
\int \dot D_{1}(u,v) g(u) \diff D(u,v)=\int \int \dot D_1(u,v) \dot D_1(u, \diff v) g(u)\diff u= \frac12\int g(u) \diff u
\end{align}
for any integrable function $g$ on $[0,1]$ and any copula $D$ satisfying Condition~\ref{cond:c1}. Indeed, if $(U,V)\sim D$, then the conditional distribution of $V$ given $U=u$ has cdf $v \mapsto \dot D_1(u,v)$. By the disintegration theorem, we obtain that
\[
\int \dot D_{1}(u,v) g(u) \diff D(u,v)
=
\int \int 
\dot D_{1}(u,v)g(u) P_{V \mid U=u}(\diff v) P_U(\diff u)
=
\int \int
\dot D_{1}(u,v)g(u) \dot D_1(u,\diff v)   \diff u.
\]
Since $v \mapsto \dot D_1(u,v)$ is continuous, we have $\int 
\dot D_{1}(u,v) \dot D_1(u,\diff v)=\Prob(X_u \le X'_u)=\frac12$, where $X_u, X'_u \sim \dot D_1(u,\cdot)$ are iid, which implies \eqref{eq:dd1}.
We apply this identity with $D=C$ and with $g(u)=\bm1(U_i \le u) - u$, and obtain
$
\int \dot C_{1}(u,v) \{\bm 1(U_i \le u)-u\} \diff C(u,v) 
=
\frac12 \int  \{\bm 1(U_i \le u)-u\} \diff u = \frac14 - \frac{U_i}2.
$
Thus
$
\int \dot C_{1}(u,v) \alpha_{n}(u,1) \diff C(u,v) 
=n^{-1/2} \sum_{i=1}^n \frac14 - \frac{U_i}2
$ and likewise,
$
\int \dot C_{2}(u,v) \alpha_{n}(1,v) \diff C(u,v) 
=n^{-1/2} \sum_{i=1}^n \frac14 - \frac{V_i}2.
$ Assembling terms implies $2\int \bar \Cb_{n} \diff C = S_{n,\{1,2\}}^\tau$ as asserted.

It remains to prove \eqref{eq:rntau}. First,
$
\int \Cb_{n,I} \diff (\hat C_{n,I}-C_{I} )
=
\int \bar \Cb_{n,I} \diff (\hat C_{n,I}-C_{I} ) + S_{n,I},
$
where $\max_{I \in \mathcal I_2(d)}|S_{n,I}|= O_\as(s_n)$ by Theorem~\ref{theo:main} and boundedness of the integration domain. Next, consider the grid $G=G_n$ defined by the points $(i/n,j/n)$ with $i,j\in\{1, \dots, n\}$. Write $\bar \Cb_{n,I}^G$ for the function on $[0,1]^2$ that is constantly equal to $\Cb_{n,I}(i/n, j/n)$ on the sets $(\frac{i-1}n, \frac{i}n] \times  (\frac{j-1}n, \frac{j}n]$ and zero elsewhere. Then,
\begin{align} \label{eq:tausn}
\int \bar \Cb_{n,I} \diff (\hat C_{n,I}-C_{I} )
=
\int (\bar \Cb_{n,I} - \bar \Cb_{n,I}^G) \diff (\hat C_{n,I}-C_{I} ) + \int \bar \Cb_{n,I}^G \diff (\hat C_{n,I}-C_{I} )
\end{align}
The first integral can be bounded by $\max_{I\in \mathcal I_2(d)} | \int (\bar \Cb_{n,I} - \bar \Cb_{n,I}^G) \diff (\hat C_{n,I}-C_{I} ) |
\le \max_{I \in \mathcal I_2(d)} 
2 \|\bar \Cb_{n,I}  - \bar \Cb_{n,I}^G\|_\infty$.
For $\bm v \in [0,1]^2$, write $\bm v_n^*$ for the closest grid point $(i/n, j/n)$ that is larger than $\bm v$ coordinatewise; note that $\bar \Cb_{n,I}^G(\bm v)=\bar \Cb_{n,I}(\bm v_n^*)$ by definition. Then, for $I=\{\ell,m\}\in\mathcal I$,
\begin{align*}
|\bar \Cb_{n,I}(\bm v) - \bar \Cb_{n,I}^G(\bm v)|
&=
|\alpha_{n,I}(\bm v)-\alpha_{n,I}(\bm v_n^*)  - \dot C_{I,\ell}(\bm v) \alpha_{n\ell}(v_1) + \dot C_{I,\ell}(\bm v_n^*) \alpha_{n\ell}(v_{n1}^*) \\
&\hspace{5cm}
-  \dot C_{I,m}(\bm v) \alpha_{nm}(v_2) + \dot C_{I,m}(\bm v_n^*) \alpha_{nm}(v_{n2}^*)|
\\&\le
\omega_{\alpha_{n,I}}(\tfrac1n) +
 |\dot C_{I,\ell}(\bm v) - \dot C_{I,\ell}(\bm v_n^*)| |\alpha_{n\ell}(v_1)| +
 \dot C_{I,\ell}(\bm v_n^*) | \alpha_{n\ell}(v_1) -\alpha_{n\ell}(v_{n1}^*) |
\\&\hspace{1.5cm}
+ |\dot C_{I,m}(\bm v) - \dot C_{I,m}(\bm v_n^*)| |\alpha_{nm}(v_2)| +
 \dot C_{I,m}(\bm v_n^*) | \alpha_{nm}(v_2) -\alpha_{nm}(v_{n2}^*) |
 \\&\le 
 3 \omega_{\alpha_{n,I}}(\tfrac1n) + \sup_{\bm v \in [0,1]^2}|\dot C_{I,\ell}(\bm v) - \dot C_{I,\ell}(\bm v_n^*)| |\alpha_{n\ell}(v_1)| 
 \\&\hspace{5cm} +  \sup_{\bm v \in [0,1]^2}|\dot C_{I,m}(\bm v) - \dot C_{I,m}(\bm v_n^*)| |\alpha_{nm}(v_2)|
\end{align*}
We have $\max_{I\in\mathcal I_2(d)}\omega_{\alpha_{nI}}(n^{-1})=O_{\as}(n^{-1/2} \log(nd))=o_\as(s_n)$, where the latter is true because $\log d \ll n$. Indeed, Equation \eqref{eq:bom2-new} with $r=n^{-1}$ and $k=2$ yields
\begin{align*}
 \Prob \Big( \max_{I \subset \{1, \dots, d\}: |I|\le 2} \omega_{\alpha_{n,I}} (n^{-1}) > \lambda\Big)
 &\le 
 M_1 nd^2 \exp\Big( - M_2 \lambda^2 n \psi\big( \lambda \sqrt n \big)\Big).
\end{align*}
Observing that $x^2 \psi(x)=(1+x)\log(1+x) - x = \log(1+x) + x \{ \log(1+x) - 1\} \ge x$ for  $x \ge e^2-1$,  the previous expression is upper bounded by
$
 M_1 nd^2 \exp( - M_2 \lambda \sqrt n)
$
for $\lambda \sqrt n \ge e^2-1$.
Setting $\lambda=c n^{-1/2}\log(nd)$, this becomes
$
 M_1 nd^2 \exp( - M_2 c \log(nd)) =  M_1 n^{1-cM_2}d^{2-cM_2},
$
which is summable in $n$ if we choose $c$ sufficiently large, and hence yields the claim by the Borel-Cantelli lemma.
The remaining terms in the penultimate display can be treated as in the proof of Lemma~\ref{lem:sn3-new}. For instance, the first supremum will be decomposed into a supremum over $\bm v$ such that either $v_1 \in [0, \delta_n) \cup(1-\delta_n, 1]$ or such that $v_1 \notin [0, \delta_n) \cup(1-\delta_n, 1]$, with $\delta_n:= 2n^{-1/2} (\log(nd))^{1/2}$.

It remains to consider the second integral on the right-hand side of \eqref{eq:tausn}. For that purpose, write $G_{i,j}=(\frac{i-1}n, \frac{i}n] \times  (\frac{j-1}n, \frac{j}n]$ and $C_I(A)$ for the $C_I$-measure of a Borel set $A$, such that
\begin{align*}
\Big| \int \bar \Cb_{n,I}^G \diff (\hat C_{n,I}-C_{I} ) \Big| 
&=
\Big|  \frac1{n^2} \sum_{i,j=1}^n \bar \Cb_{n,I}^G(i/n, j/n) \times (\hat C_{n,I}-C_I)(G_{i,j}) \Big| 
\\&\le
\|\bar \Cb_{n,I}\|_\infty \times 4  \| \hat C_{n,I} - C_I\|_\infty 
\\&\le
4n^{-1/2} \max_{I \in \mathcal I_2(d)}\|\bar \Cb_{n,I}\|_\infty \| \Cb_{n,I} \|_\infty = O_\as(n^{-1/2} \log(nd))=o_\as(s_n)
\end{align*} 
by Corollary~\ref{corr:supcn}. Overall, this implies \eqref{eq:rntau}.
\end{proof}

\begin{proof}[Proof of Theorem~\ref{theo:fwr}]
The result follows from Theorem 5.1 in \cite{Che13}. Using their notation, let $\Delta_2 := \max_{|I|=2} \frac1n \sum_{i=1}^n (\hat x_{i,I}- x_{i,I})^2$.
Note that, since $\int \Pi_2 \diff \hat C_{n,I} = \int \hat C_{n,I} \diff \Pi_2 + \frac1n$, for $I=\{\ell,m\}$,
\begin{align*}
\frac1{12}( \hat x_{i ,I}- x_{i,I}  ) 
&=
(1-\hat U_{i\ell})(U_{im}-\hat U_{im})
+
(1-U_{im})(U_{i\ell}- \hat U_{i\ell})  - 3 \int (\hat C_{n,I} - C_I) \diff \Pi_2 + \frac3n 
\\&\hspace{2.5cm} +\int_0^1 \{ \hat C_{n,I}(\hat U_{i\ell}, z) - C_I(\hat U_{i\ell},z) \} + \{ C_I(\hat U_{i\ell},z) - C_I(U_{i\ell},z)\} \diff z 
\\&\hspace{2.5cm} +\int_0^1 \{ \hat C_{n,I}(z,\hat U_{im}) - C_I(z,\hat U_{im}) \} + \{ C_I(z,\hat U_{im}) - C_I(z,U_{im})\} \diff z.
\end{align*}
Since $C_I$ is Lipschitz-continuous, we obtain the bound
\begin{align*}
\max_{I\in \mathcal I_2(d)}| \hat x_{i ,I}- x_{i,I}  | 
&\le \frac{36}n + 
\frac{12}{\sqrt n}\max_{I=\{\ell, m\} \subset \{1, \dots, d\}}\{ 2\| \alpha_{n\ell}\|_\infty +  2\| \alpha_{nm}\|_\infty + 5 \|\Cb_{n,I}\|_\infty\}
\\&=
O_\as(n^{-1/2} \log(nd)^{1/2}),
\end{align*}
where we made use of Proposition~\ref{prop:alphauni-all} and Corollary~\ref{corr:supcn}.
It follows that
\[
\Delta_2 := \max_{I\in \mathcal I_2(d)} \frac1n \sum_{i=1}^n (\hat x_{i ,I}- x_{i,I})^2 =O_\as(n^{-1}\log(nd))
\]
almost surely. Next, again using the same notation as in \cite{Che13}, let $\Delta_1 = \max_{I\in \mathcal I_2(d)}| \sqrt n (\hat \rho_{n,I} - \rho_I) - n^{-1/2} \sum_{i=1}^n x_{i,I} |$, which can be bounded by $O_{\as}(n^{-1/4}(\log(nd))^{3/4})$ by
Proposition~\ref{prop:asso}. Note that the rate holds uniformly in $(C^{(n)})_n \in (\Gamma^{(n)})_n$, because we fixed $K>0$ in advance. 

By our assumption on $d=d_n$, we obtain that Condition (M) in \cite{Che13} is met (with $p=d(d-1)/2$, $B_n$ sufficiently large and constant in $n$, $c_2$ sufficiently small and $C_2$ sufficiently large), uniformly in $(C^{(n)})_n\in (\Gamma^{(n)})_n$. Hence, by Theorem 5.1 in that reference, \eqref{eq:fer} is met.
\end{proof}

\begin{proof}[Proof of Proposition~\ref{prop:moeb}]
We start by proving \eqref{eq:moeb0} for fixed $k\in\N$. Note that $(\prod_{j=1}^n a_j) - (\prod_{j=1}^n b_j) = \sum_{\ell=1}^n (a_\ell - b_\ell) (\prod_{j=1}^{\ell-1} a_j)(\prod_{j=\ell+1}^n b_j)$ for any real numbers $a_1, \dots, a_n$, $b_1, \dots, b_n$. Hence, since $\hat C_{n,j}(u_j) = \lfloor nu_j \rfloor/n$, we have
\begin{align*}
\Mc_I(\hat C_n)  - \widebar \Mc_I(\hat C_n)
&=
\sum_{B \subset I} (-1)^{|I \setminus B|} \hat C_n(\bm u^{B}) \Big[ \prod_{j \in I \setminus B} \hat C_{n,j}(u_j) - \prod_{j \in I \setminus B} u_j\Big] \\
&=
\sum_{B \subset I} (-1)^{|I \setminus B|} \hat C_n(\bm u^{B}) \sum_{\ell \in I \setminus B} \Big(\frac{\lfloor nu_\ell \rfloor}n - u_\ell\Big) \Big(\prod_{j \in I \setminus B: j < \ell} u_j\Big) \Big(\prod_{j \in I \setminus B: j > \ell}\frac{\lfloor nu_j \rfloor}n\Big). 
\end{align*}
This implies \eqref{eq:moeb0} in view of the fact that $\sup_{u\in[0,1]}| \lfloor nu \rfloor/n -u|\le n^{-1}$ and $|\{B:B \subset I\}| \le 2^k$ for any $I$ with $2 \le |I| \le k$.

Subsequently, let $k=2$. As argued before Proposition~\ref{prop:moeb}, it is sufficient to show \eqref{eq:sni} with $S_{n,I}$ replaced by $\bar S_{n,I}$. Recall the decomposition $\bar S_{n,i}=U_{n,I} + V_{n,I}$, see \eqref{eq:univni}, and suppose we have shown that
\begin{align} \label{eq:gumbelapp-u}
\lim_{n\to\infty}\Prob\Big( \pi^4 \max_{I\in \mathcal I_2(d)}U_{n,I} - u_n \le y \Big) = \exp\Big\{ - \Big(\frac{\kappa^2}{8\pi}\Big)^{1/2} \exp\Big(-\frac{y}2\Big) \Big\},
\end{align}
and
\begin{align} \label{eq:vno}
\max_{I\in \mathcal I_2(d)} |V_{n,I}| =o_\Prob(1).
\end{align}
The assertion in \eqref{eq:sni} with $S_{n,I}$ replaced by $\bar S_{n,I}$ then follows from
\[
\pi^4\max_{I\in \mathcal I_2(d)} S_{n,I} - u_n \le \{ \pi^4 \max_{I\in \mathcal I_2(d)} U_{n,I} - u_n \} + \pi^4 \max_{I\in \mathcal I_2(d)} |V_{n,I}| ,
\]
and 
\[
\pi^4 \max_{I\in \mathcal I_2(d)} S_{n,I} - u_n \ge \{ \pi^4 \max_{I\in \mathcal I_2(d)} U_{n,I} - u_n \} - \pi^4 \max_{I\in \mathcal I_2(d)} |V_{n,I}|  .
\]

It remains to prove \eqref{eq:gumbelapp-u} and \eqref{eq:vno}, and we start with \eqref{eq:gumbelapp-u}. We will apply Theorem 4.2 in \cite{Drt20}. Identifying objects with those needed in that theorem and writing $I=\{j_1, j_2\}$, we have 
\[
h_I(\bm u, \bm v) 
= h_I((u_{j_1}, u_{j_2}), (v_{j_1}, v_{j_2})) 
= 
\frac{4}{\pi^{4}} \sum_{\ell_1,\ell_2 \in \N } \frac1{(\ell_1 \ell_2)^2} \prod_{s=1}^2 \cos(\ell_s \pi  u_{j_s}) \cos(\ell_s \pi v_{j_s}),
\]
where $\lambda_{\ell_1, \ell_2}=\frac1{(\pi^2 \ell_1 \ell_2)^2}$ are the eigenvalues and $\phi_{\ell_1, \ell_2}(u_1, u_2) = 2 \cos(\ell_1 \pi u_1) \cos(\ell_2 \pi u_2)$ the corresponding uniformly bounded eigenfunctions
of the integral operator that maps a function $g$ to the function $\bm u \mapsto \int h_I(\bm u, \bm v)g(\bm v) \diff \Pi(\bm v)$. Note that, up to a factor, this is the same expansion as in Examples 2.1-2.3 in \cite{Drt20}. The largest eigenvalue is $\lambda_1 := \lambda_{1,1} = \pi^{-4}$ (with multiplicity $\mu_1=1$) and the sum over all eigenvalues is $\Lambda := \sum_{\ell_1, \ell_2 \in \N} \frac1{(\pi^2 \ell_1 \ell_2)^2} = \frac{1}{36}$. Moreover, 
\[
\kappa^2 = \prod_{(\ell_1, \ell_2) \ne (1,1)} \Big(1- \frac{\lambda_{\ell_1, \ell_2}}{\lambda_{1,1}}\Big)^{-1} 
=\prod_{(\ell_1, \ell_2) \ne (1,1)}\Big( 1-\frac1{\ell_1^2\ell_2^2}\Big)^{-1}
= 2 \prod_{n=2}^\infty \frac{\pi/n}{\sin(\pi/n)},
\]
where we used that $\sin(\pi /n )/(\pi /n )=\sinc(\pi/n) = \prod_{k=1}^\infty (1-\frac1{k^2n^2})$ for $n \in \N$ and that $\prod_{k=2}^\infty (1-\frac1{k^2})=1/2$.
The assertion in \eqref{eq:gumbelapp-u} then follows from Theorem 4.2 in \cite{Drt20}, with the condition on $d=d_n$ derived from the proof of Corollary 4.1 in that reference.

It remains to prove \eqref{eq:vno}. A simple calculation shows that, for $I,J \in \Ic_2(d)$ and $i \in\{1, \dots, n\}$,
\[
\Cov_C(h_I(\bm U_i, \bm U_i), h_J(\bm U_i, \bm U_i))= \frac1{90} \bm 1(I=J)
\]
(which actually holds for all $C$ such that $C_I=\Pi_4$ for all $I \in \Ic_4$).
Let 
$
(Y_{d,I})_{I \in \mathcal I_2}$ be $\mathcal N_{d(d-1)/2}(0, \frac1{90}\bm I_{d(d-1)/2})
$-distributed,
where $\bm I_{d(d-1)/2}$ is the $d(d-1)/2$-dimensional identity matrix. 
Since $d=d_n$ satisfies $\log d =o(n^{1/5})$ as a consequence of our assumption on $d$, Theorem 2.1 in \cite{Che22} implies that
\begin{align*} 
\lim_{n\to\infty}\sup_{t \in \R} | \Prob(\sqrt n  \max_{I\in \mathcal I_2(d)} | V_{n,I} | \le t) - \Prob(\max_{I\in \mathcal I_2(d)} |Y_{d,I}| \le t) | = 0.
\end{align*}
As in the proof of Corollary~\ref{cor:gaussapp}, Lemma 1 in \cite{Deo72}  implies that
\begin{align*} 
\lim_{n \to \infty}\Prob\Big[ \sqrt{2 \log c_n}\Big\{  \sqrt{90n} \max_{I\in \mathcal I_2(d)} | V_{n,I} | - \Big( \sqrt{2 \log c_n} - \frac{ \log(4\pi \log c_n)-4}{2 \sqrt{2\log c_n}} \Big) \Big\}  \le t \Big] = e^{-e^{-t}}
\end{align*}
for all $t\in\R$. The latter straightforwardly implies \eqref{eq:vno}.
\end{proof}

\appendix

\section{Proofs for Example~\ref*{ex:copulas}}
\label{sec:example-proofs}

We start by mentioning the following useful observation: for $d$-dimensional copulas with $d\ge 3$, the mixed second order partial derivatives are bounded by bivariate marginal copula densities. For instance,
\begin{align*}
\ddot C_{12}(u_1, u_2, u_3) 
&= 
\Prob(U_3 \le u_3 \mid U_1 = u_1, U_2 = u_2)
c_{12}(u_1, u_2)
\end{align*}
where $c_{12}$ denotes the density of $(U_1, U_2)$. As a consequence, suitable bounds on the bivariate marginal densities carry over to bounds on the mixed-second order partial derivatives.

\subsection{The multivariate Gaussian copula} 
For a full-rank correlation matrix $\Sigma=(\rho_{i,j})_{i,j=1, \dots, d}$, the $d$-variate Gaussian copula is given by
\[
C(\bm u) = \Phi_\Sigma( x_1, \dots, x_d),
\]
where $x_j=\Phi^{-1}(u_j)$. In view of the fact that the $k$-variate margins of the Gaussian copula are Gaussian copulas again (with the respective correlation matrix obtained by suitable omission of rows and columns), it is sufficient to show the claim for $k=d$. The first order partial derivatives are continuous as argued in Example 5.1 in \cite{Seg12}, so it remains to consider the second order partial derivatives.

We start with the case $d=2$, and write $\rho \in (-1,1)$ for the correlation parameter. The respective bivariate Gaussian copula density is then given by $c_\rho(u,v) = \ddot C_{12}(u,v) = \varphi_\rho(x,y) / \{\varphi(x) \varphi(y)\}$, where $x=\Phi^{-1}(u)$ and $ y=\Phi^{-1}(v)$. A straightforward calculation shows that, for $(u,v) \in (0,1)^2$,
\begin{align} \label{eq:gauss-copula-density}
c_\rho(u,v) = \frac1{\sqrt{1-\rho^2}} \exp\Big[ - \frac14 \frac\rho{1-\rho^2} \big\{ (1+\rho) (x-y)^2 - (1-\rho)(x+y)^2\big\} \Big],
\end{align}
which is clearly continuous.

For deriving bounds on $c_\rho(u,v)$, we restrict attention to the case $\rho>0$, and consider the lower tail where $u,v$ are close to 0; the other cases can be treated by similar arguments.
Consider the term in curly brackets in \eqref{eq:gauss-copula-density}. It can be rewritten as
\begin{align*}
(1+\rho) (x-y)^2 - (1-\rho)(x+y)^2
&=
2\rho (x^2+y^2) - 4 xy
\\ &=
\frac2{\rho} \big\{ (\rho x - y)^2 - y^2(1-\rho^2) \big\} 
\ge 
- \frac{2(1-\rho^2)}{\rho} y^2;
\end{align*}
note that the inequality is sharp for $x=y/\rho$.
Using the bound $|\Phi^{-1}(u)| = \Phi^{-1}(1-u) \le \sqrt{2\log(1/u)}$ for $u \le 1/2$ (see, e.g., Proposition 4.1 in \citealp{BouTho12}), we obtain
\[
c_\rho(u,v) \le \frac1{\sqrt{1-\rho^2}} \exp\Big\{ \frac14 \frac\rho{1-\rho^2} \frac{2(1-\rho^2)}{\rho} y^2 \Big\}
=
 \frac1{\sqrt{1-\rho^2}} e^{y^2/2} 
 \le 
 \frac1{\sqrt{1-\rho^2}} \frac1v
\]
for $0<v \le 1/2$.
Interchanging the roles of $x$ and $y$, we obtain that
$c_\rho(u,v) \le (1-\rho^2)^{-1/2}(u \vee v)^{-1} \le (1-\rho^2)^{-1/2}(u(1-u) \vee v(1-v))^{-1}$ for all $0<u,v\le 1/2$, as required in Condition~\ref{cond:c2-new}.

Next, consider the second order partial derivative $\ddot C_{jj}$; clearly, it suffices to consider $j=1$. For $u\in(0,1)$, we have
\[
\ddot C_{11}(u,v) = 
\begin{dcases}
- \frac{\rho}{\sqrt{1-\rho^2}} \varphi\Big( \frac{y-\rho x}{\sqrt{1-\rho^2} } \Big) \frac{1}{\varphi(x)}, & v \in (0,1)\\
0 & v \in \{0,1\}
\end{dcases}
\]
see Formula (58) in \cite{OmeGijVer09} (or the calculations below in the case $d\ge 3$). This function is clearly continuous on $(0,1)\times [0,1]$, and since $\varphi(x) \ge u/\sqrt{2\pi}$ for $u\in[0,1/2]$, we obtain the bound
\[
|\ddot C_{11}(\bm u)| \le \frac{|\rho|}{\sqrt{1-\rho^2}} \frac1u
\le \frac{|\rho|}{\sqrt{1-\rho^2}} \frac1{u(1-u)}
\]
for $u\le 1/2$.
Overall, the bivariate Gaussian copula satisfies Condition~\ref*{cond:c2-new} if $|\rho|<1$, and the multivariate Gaussian satisfies Condition~\ref*{cond:c2-new} with $k=2$ if $\max_{i \ne j}| \rho_{ij}| \le \rho_0$ for some $\rho_0<1$, with $K= |\rho_0|(1-\rho_0^2)^{-1/2}$.

Next, consider the case $d\ge 3$. For $j\in\{1, \dots, d\}$ and  $u_j \in (0,1)$, the definition of $C$ yields
\[
\dot C_j(\bm u) = \dot\Phi_{\Sigma,j}( x_1, \dots, x_d) \frac{1}{\varphi(x_j)}
=
\Prob(\bm X_{-j} \le \bm x_{-j}\mid X_j = x_j),
\]
where $\bm X \sim \Nc(0, \bm \Sigma)$. Here, the last equation holds because
\begin{align*}
\dot\Phi_{\Sigma,1}( x_1, \dots, x_d) 
&=
\lim_{h \downarrow 0} h^{-1} \Big\{ \Prob(X_1 \le x_1+h, \bm X_{-1} \le \bm x_{-1}) -  \Prob(X_1 \le x_1, \bm X_{-1} \le \bm x_{-1}) \Big\}
\\&=
\lim_{h \downarrow 0}\frac{\Prob(X_1 \in (x_1, x_1+h], \bm X_{-1} \le \bm x_{-1})}{\Prob(X_1 \in (x_1, x_1+h])}  \frac{\Prob(X_1 \in (x_1, x_1+h])}{h}
\\&=
\lim_{h \downarrow 0}\Prob(\bm X_{-1} \le \bm x_{-1} \mid X_1 \in (x_1, x_1+h])\frac{\Phi(x_1+h)-\Phi(x_1)}{h}
\\&=
\Prob(\bm X_{-1} \le \bm x_{-1} \mid X_1 = x_1) \varphi(x_1).
\end{align*}
Write $\Sigma^{(-j)}\in \R^{(d-1) \times (d-1)}$ for the matrix obtained by deleting the $j$th row and column from $\Sigma$, and write $\sigma^{(j)} =(\rho_{ij})_{i \ne j}\in \R^{d-1}$ for the deleted column without the $j$th entry. 
By the properties of normal distributions, we have 
\[
(\bm X_{-j} \mid X_{j} = x_j) =_d \Nc_{d-1}( x_j\sigma^{(j)} , \Xi^{(j)}),
\]
$\Xi^{(j)}= \Sigma^{(-j)}-\sigma^{(j)} (\sigma^{(j)})^\top $.
As a consequence, we may write, for $\bm u \in V_j$,
\begin{align}
\label{eq:gauss-d1}
\dot C_j(\bm u) 
=
\Phi_{\Xi^{(j)}}(\bm x_{-j} - x_j\sigma^{(j)}).
\end{align}
Hence, for $\bm u \in V_j$ such that all coordinates of $\bm u_{-i}$ are non-zero and such that $\bm u_{-j} \ne \bm 1 \in  \R^{d-1}$,
\[
\ddot C_{jj}(\bm u) = - \sum_{i \ne j} \dot \Phi_{\Xi^{(j)}, i}(\bm x_{-j} - x_j\sigma^{(j)}) \frac{ \rho_{j,i} }{\varphi(x_j)},
\]
while elementary calculations show that $\ddot C_{jj}(\bm u)=0$ if some coordinates of $\bm u_{-j}$ are zero or if $\bm u_{-1}=\bm 1$. 
Overall, writing $W= (0,1]^{d-1} \setminus\{ \bm 1\}$, we have, for all $\bm u \in V_j$,
\[
\ddot C_{jj}(\bm u) = 
\begin{dcases}
- \sum_{i \ne j} \dot \Phi_{\Xi^{(j)}, i}(\bm x_{-j} - x_j\sigma^{(j)}) \frac{ \rho_{ij} }{\varphi(x_j)}, & \bm u_{-j} \in W \\
0 & \bm u_{-j} \in W^c.
\end{dcases}
\]
Similar as before, we may write 
\[
\ddot \Phi_{\Xi^{(j)}, i}(\bm z) = \Prob(\bm Z_{-i}^{(j)} \le  \bm z_{-i} \mid Z_i^{(j)} = z_i) f_{Z_i^{(j)}}(z_i)
= H_{ij}(\bm z_{-i} \mid z_i)  \varphi\Big( \frac{z_i }{ (1-\rho_{ij}^2 )^{1/2} }\Big) \frac{1}{  (1-\rho_{ij}^2 )^{1/2}}
\]
where $\bm Z^{(j)} \sim \mathcal N_{d-1}(\bm 0, \Xi^{(j)})$, where $H_{ij}(\bm z_{-i} \mid z_i) = \Prob(\bm Z_{-i}^{(j)} \le  \bm z_{-i} \mid Z_i^{(j)} = z_i) $ and where $f_{Z_i^{(j)}}$ is the density of $Z_i^{(j)}$; note that $\Var(Z_i^{(j)}) = \Xi^{(-j)}_{ii} =1-\rho_{ij}^2$. Hence, for $\bm u \in V_j$ such that $\bm u_{-j} \in W$,
\[
\ddot C_{jj}(\bm u)
= 
- \sum_{i \ne j} \rho_{ij} H_{ij}(\bm x_{-j} - x_j\sigma^{(j)} \mid x_i - x_j \rho_{ij}) \frac{  \varphi\big( (x_i - x_j \rho_{ij})/ (1-\rho_{ij}^2)^{1/2}\big)}{  (1-\rho_{ij}^2)^{1/2} \varphi(x_j)}.
\]
If $\bm u_{-j} \to \bm 1$, we have $x_i \to \infty$ for all $i\ne j$, and the previous expression goes to zero, since $|H_{ij}| \le 1$. If some coordinates of $\bm u_{-j}$ go to zero, say, only the $k$th one, then $x_k \to -\infty$, and hence both $  \varphi( (x_i - x_j \rho_{ij})/(1-\rho_{ij}^2)^{1/2})$ and $H_{ij}(\bm x_{-j} - x_j\sigma^{(j)} \mid x_i - x_j \rho_{ij})$ go to zero for all $i\ne k$. This yields continuity of $\ddot C_{jj}$ on $V_j$. Moreover, since $\varphi(x_j) \ge u_j/\sqrt{2\pi}$ for $u_j$ near zero, we get that 
\[
\big| \ddot C_{jj}(\bm u) \big| 
\le 
u_j^{-1} \sum_{i \ne j} \Big| \frac{\rho_{ij}}{ (1-\rho_{ij}^2)^{1/2}} \Big| \bm1(u_i < 1)
= 
u_j^{-1} \sum_{i \ne j} \Big| \frac{\rho_{ij}^2}{ 1-\rho_{ij}^2} \Big|^{1/2} \bm1(u_i < 1)
\]
(note that we retrieve the result for $d=2$ from above).

It remains to treat the mixed partial derivatives $\ddot C_{ij}$, for which the bounds follow from the established bivariate bounds by the argument given in the first paragraph of this section. 
We only need to show continuity on $V_i \cap V_j$. From \eqref{eq:gauss-d1}, we get, for all $\bm u \in V_i \cap V_j$,
\[
\ddot C_{ij}(\bm u) = \frac{\dot \Phi_{\Xi^{(j)}, i}(\bm x_{-j} - x_j \sigma^{(j)})}{\varphi(x_i)}
=
H_{ij}(\bm x_{-j} - x_j \sigma^{(j)} \mid x_i - \rho_{ij}x_j)\frac{\varphi(x_i - \rho_{ij}x_j)}{\varphi(x_i)}
\]

Overall, if there exists $\rho_0<1$ such that $\max_{i\ne j} |\rho_{ij}| \le \rho_0$, we obtain that Condition 2.3 is met with $K=(d-1) \{ \rho_0^2/(1-\rho_0^2) \}^{1/2}$.

\subsection{The multivariate Hüsler-Reiss copula} 
The multivariate Hüsler-Reiss copula \citep{Eng15} has the important property that bivariate margins are again Hüsler-Reiss. Specifically, with $\lambda=\lambda_{ij}$ the parameter controlling the $(i,j)$-margin, the bivariate Hüsler-Reiss copula is given by 
\[
C_\lambda(u,v) = \exp\Big\{ \log(uv) A_\lambda \Big( \frac{\log v}{\log(uv)} \Big) \Big\}
\]
with Pickands dependence function
\[
A_\lambda(t) = (1-t) \Phi\Big( \lambda + \frac1{2\lambda} \log\Big( \frac{1-t}{t} \Big) \Big) + t \Phi\Big( \lambda + \frac1{2\lambda} \log\Big( \frac{t}{1-t} \Big) \Big), \qquad t \in [0,1],
\]
with parameter $\lambda \in [0,\infty]$ and with $\Phi$ the cdf of the standard normal distribution \citep{GudSeg10}. Note that $\lambda=0$ corresponds to complete dependence, and $\lambda=\infty$ corresponds to independence. Hence, it is sufficient to consider $\lambda \in (0,\infty)$. In view of Example 5.3 in \cite{Seg12}, $C_\lambda$ satisfies Condition~\ref{cond:c1} with $K=1+M(\lambda)$ if
\[
M(\lambda)
:=
\sup_{t \in (0,1)} t(1-t) A_\lambda''(t) < \infty.
\]
Subsequently, we will show that $M(\lambda) \le (L/4) (\lambda^{-1}+\lambda^{-2})$ for some universal constant $L$, which implies that $C_\lambda$ satisfies Condition~\ref{cond:c1} for any $\lambda>0$, and that the multivariate Hüsler-Reiss copula satisfies Condition~\ref{cond:c2-new} with $k=2$ and $K=(L/4) (\lambda_0^{-1}+\lambda_0^{-2})$, provided that $\min_{i\ne j} \lambda_{i,j} \ge \lambda_0>0$.

Write $c_{\lambda}(t) = (2\lambda)^{-1} \log((1-t)/t)$, such that 
\[
A_\lambda(t) = (1-t) \Phi\big( \lambda + c_{\lambda}(t) \big) + t \Phi\big( \lambda -c_{\lambda}(t)  \big).
\]
Note that $c_{\lambda}'(t) = - \{ 2\lambda t (1-t)\}^{-1} $. Hence, 
\begin{align*}
    A_\lambda'(t) 
    = 
    -  \Phi\big( \lambda + c_{\lambda}(t) \big) 
    -  \frac1{2\lambda t} \varphi\big( \lambda + c_{\lambda}(t) \big) 
    +  \Phi\big( \lambda - c_{\lambda}(t) \big)
    + \frac1{2\lambda (1-t)}  \varphi\big( \lambda - c_{\lambda}(t) \big) ,
\end{align*}
with $\varphi=\Phi'$ the standard normal density.
Next, observe that $\varphi'(t)=-t\varphi(t)$. Hence,
\begin{align*}
    A_\lambda''(t) 
    &= 
    \frac1{2\lambda t(1-t)} \varphi\big( \lambda + c_{\lambda}(t) \big) 
    + \frac1{2\lambda t^2} \varphi\big( \lambda + c_{\lambda}(t) \big)  \Big\{ 1 - \frac{\lambda + c_\lambda(t)}{2 \lambda  (1-t)} \Big\}  
    \\& \hspace{.5cm}
    + \frac1{2\lambda t(1-t)} \varphi\big( \lambda - c_{\lambda}(t) \big) 
    + \frac1{2\lambda (1-t)^2} \varphi\big( \lambda - c_{\lambda}(t) \big)  \Big\{ 1 - \frac{\lambda - c_\lambda(t)}{2 \lambda  t} \Big\}  
    \\ &=
    \varphi\big( \lambda + c_{\lambda}(t) \big) \frac{\lambda - c_\lambda(t)}{4 \lambda^2 t^2 (1-t)} + \varphi\big( \lambda - c_{\lambda}(t) \big) \frac{\lambda + c_\lambda(t)}{4 \lambda^2 t (1-t)^2}.
\end{align*}
Hence,
\begin{align*}
    t(1-t) A_\lambda''(t) 
    &=
    \frac{1}{4\lambda^2} \Big\{ \frac{\lambda - c_\lambda(t)} t \varphi\big( \lambda + c_{\lambda}(t) \big)+ \frac{\lambda + c_\lambda(t)}{1-t} \varphi\big( \lambda - c_{\lambda}(t) \big) \Big\}.
\end{align*}
Define
\[
g_\lambda(t) := \frac{\lambda - c_\lambda(t)} t \varphi\big( \lambda + c_{\lambda}(t) \big).
\]
Observing that $c_\lambda(1-t) = - c_\lambda(t)$, we may write
\begin{align} \label{eq:expression-app}
t(1-t) A''(t) = \frac{1}{4\lambda^2} \{ g_\lambda(t) + g_\lambda(1-t) \}.
\end{align}
It is hence sufficient to show that $g_\lambda(t)$ is bounded from above on $(0,1)$. 
For that purpose, define $\ell_\lambda = 1/\{1+\exp(2 \lambda^2)\}$ and $u_\lambda = 1/\{1+\exp(-2 \lambda^2)\}$, and note that $0 < \ell_\lambda < 1/2 < u_\lambda < 1$. 
Further, we note that $t \mapsto c_\lambda(t)$ is decreasing with $\lim_{t\to 0} c_\lambda(t)=\infty$, $\lim_{t\to 1} c_\lambda(t)=-\infty$ and that $t\le 1/2$ is equivalent to $c_\lambda(t) \ge 0$. 

We now distinguish four cases: 
\begin{compactitem}
\item 
For $t \in  (0, \ell_\lambda]$, we have $\lambda\le c_\lambda(t)$, and hence $g_\lambda(t) \le 0$. 
\item 
For $t\in(\ell_\lambda,1/2]$, we have $\lambda \ge c_\lambda(t) \ge 0 $, and hence
\begin{align*}
\lambda + c_{\lambda}(t) 
\ge 2c_\lambda(t) = \lambda^{-1} \log((1-t)/t) \ge \lambda^{-1} \log(1/(2t)) 
&\ge \lambda^{-1} \log\big(\tfrac12 \{ 1+\exp(2\lambda^2) \} \big)
\\&\ge 
\lambda^{-1} \log\big(\tfrac12 \exp(2\lambda^2) \big).
\end{align*}
Therefore, since $c_\lambda(t) \ge 0$ and since $\varphi$ is decreasing on $[0,\infty)$
\begin{align*}
g_\lambda(t) 
\le 
\frac{\lambda}{s_\lambda(t)} \varphi\big( \lambda^{-1} \log\big(\tfrac12 \exp(2\lambda^2) \big) \big)
&\le 
\lambda \exp(2\lambda^2) \frac1{\sqrt{2\pi}} \exp\Big(- \frac1{2\lambda^{2}} \big( 2\lambda^2 - \log 2\big)^2 \Big)
\\&=
\lambda\frac1{\sqrt{2\pi}} \exp\Big( 2 \log 2 - \frac{\log^2 2}{2\lambda^2}\Big) 
\\&\le 
\lambda\frac{e^4}{\sqrt{2\pi}}. 
\end{align*}
\item 
For $t \in (1/2, u_\lambda)$, we have $-\lambda < c_\lambda(t) <0$, which implies $\lambda-c_\lambda(t) \le 2 \lambda$ and hence, since $t\ge 1/2$,
\begin{align*}
g_\lambda(t) 
= 4 \lambda
\varphi\big( \lambda + c_{\lambda}(t) \big) 
\le \lambda \frac{4}{\sqrt{2\pi}}.
\end{align*}

\item For $t \in[ u_\lambda,1)$, we have $c_\lambda(t) \le - \lambda<0$. Hence, since $t\ge 1/2$,
\begin{align*}
g_\lambda(t) 
= \frac{\lambda + |c_\lambda(t)|} t \varphi\big( \lambda - |c_{\lambda}(t)| \big)
&\le 
 2\big(\lambda + |c_\lambda(t)|\big) \varphi\big( \lambda - |c_{\lambda}(t)| \big)
\\&=
2 \big(|c_\lambda(t)|-\lambda+2\lambda\big) \varphi\big( |c_{\lambda}(t)|-\lambda \big)
\\&\le
2 + \lambda \frac{4}{\sqrt{2\pi}},
\end{align*}
where we used that $\varphi$ is symmetric and bounded by $1/\sqrt{2\pi}$, and that $x \varphi(x)$ is bounded by~1.
\end{compactitem}
Overall, we have shown that there exists a universal constant $L$ such that
\[
g_\lambda(t) \le L(1+\lambda)
\]
for all $t\in(0,1)$. In view of \eqref{eq:expression-app}, we hence obtain that 
\[
M(\lambda) =  \sup_{t\in(0,1)} t(1-t) A''(t)
\le \frac{L}{4} (\lambda^{-1}+\lambda^{-2}
\]
as asserted.

\subsection{The Clayton copula} We start by some generalities on Archimedean copulas.
Recall that a copula $C$ is Archimedean if
\[
C(\bm u) =\psi\big( \psi^{-1}(u_1) + \dots + \psi^{-1}(u_d)\big)
\]
for some Archimedean generator $\psi$, that is, for a function $\psi:[0,\infty] \to [0,1]$ that is nonincreasing, continuous, satisfies $\psi(0)=1$ and $\psi(\infty) = 0$ and is strictly decreasing on $[0, \inf\{x: \psi(x)=0\})$. Her, the inverse $\psi^{-1}(x)$ is the usual inverse for $x\in(0,1]$ and $\psi^{-1}(0)=\inf\{u:\psi(u)=0\}$. As shown in \cite{McnNes09}, $C$ defined as in the previous display is a copula if and only if $\psi$ is $d$-monotone on $[0,\infty)$.

Subsequently, we write $\phi=\psi^{-1}$, such that
\[
C(\bm u) =\phi^{-1}\big( \phi(u_1) + \dots + \phi(u_d)\big),
\]
and we restrict attention to the case where $\phi$ is twice continuously differentiable on $(0,1)$ with $\phi(0+)=\infty$. Clearly, $\dot C_i(\bm u) = \ddot C_{ii}(\bm u)=\ddot C_{ij}(\bm u)=0$ for $\bm u \in [0,1]^d$ such that $C(\bm u)=0$. 
If $C(\bm u)>0$, we have
\begin{align*}
\dot C_i(\bm u) &= \frac{\phi'(u_i)}{\phi'(C(\bm u))}, & & \bm u \in V_i,\\
\ddot C_{ii}(\bm u) &= \frac{\phi''(u_i)}{\phi'(C(\bm u))} - \frac{\{\phi'(u_i)\}^2 \phi''(C(\bm u))}{\{\phi'(C(\bm u))\}^3}, & & \bm u \in V_i,\\
\ddot C_{ij}(\bm u) &= -\frac{\phi'(u_i) \phi'(u_j) \phi''(C(\bm u))}{\{ \phi'(C(\bm u))\}^{3}},& & \bm u \in V_i \cap V_j.
\end{align*}

From now on, we restrict attention to the Clayton copula with parameter $\theta>0$, for which $\phi(u)=\theta^{-1}(t^{-\theta}-1)$, $\phi'(u)=-u^{-\theta-1}$ and $\phi''(u) = (\theta+1)u^{-\theta-2}$; other families can treated similarly. It follows that $|\phi''(u)/\phi'(u)| = (\theta+1) u^{-1}$, such that, for $\bm u \in V_i$,
\[
\Big| \frac{\phi''(u_i)}{\phi'(C(\bm u))} \Big|
=
\Big| \frac{\phi''(u_i)}{\phi'(u_i)} \Big| \cdot \dot C_i(\bm u)
\le 
(\theta+1) \frac1u_i.
\]
Next, since $0<C(\bm u)\le u_i$
\[
\Big| \frac{\{\phi'(u_i)\}^2 \phi''(C(\bm u))}{\{\phi'(C(\bm u))\}^3} \Big|
= (\theta+1) \frac{\{C(\bm u)\}^{2\theta+1}}{u_i^{2 \theta + 2}}
\le
(\theta+1) \frac{\{C(\bm u)\}^{2\theta+1}}{u_i^{2 \theta + 2}}
\le
(\theta+1)\frac1{u_i}.
\]
Finally, for $\bm u \in V_i \cap V_j$, since $0<C(\bm u)\le u_i \wedge u_j$,
\begin{align*}
    \Big|\frac{\phi'(u_i) \phi'(u_j) \phi''(C(\bm u))}{\{ \phi'(C(\bm u))\}^{3}} \Big|
&=
(\theta+1) \frac{\{C(\bm u)\}^{2\theta+1}}{u_i^{\theta + 1}u_j^{\theta+1}}
\le
(\theta+1)  \frac{(u_i \wedge u_j)^{2\theta+1}}{(u_i \wedge u_j)^{2\theta + 2}}
=
(\theta+1)  \frac1{u_i \wedge u_j}
\end{align*}
The preceding four displays imply that the bound in Condition 2.3 is met with $k=d$ and $K=\theta+1$. Moreover, since $\phi'(0)=\infty$, we also observe that the continuity condition on $\ddot C_{ii}$ and $\ddot C_{ij}$ is met for $k=d$, even at points $\bm u$ with $C(\bm u)=0$.

\section{Details for the proof of Proposition~\ref*{prop:asso}}

In the proof of Proposition~\ref*{prop:asso}, we claimed that
$\hat \rho_{n,I} =\tilde \rho_{n,I} + r_{n,I}^\rho$ and $\hat \tau_{n,I}=\tilde \tau_{n,I} + r_{n,I}^\tau$, where
\begin{align*}
\tilde \rho_{n,I} =
12 \int \hat C_{n,I} \diff \Pi_2 - 3, \qquad 
\tilde \tau_{n,I} = 4 \int \hat C_{n,I} \diff \hat C_{n,I} - 1,
\end{align*}
and where $|r_{n,I}^\rho| \le 6/(n-1)$ and $|r_{n,I}^\tau| \le 4/(n-1)$. We provide a detailed proof, using the formula $\int D \diff \hat C_{n,I} = \int \hat C_{n,I} \diff D + \frac1n$ that was shown to be valid for all bivariate copulas $D$ and all $I\subset \{1, \dots, d\}$ with $|I|=2$.

For Spearman's rho, using that $\sum_{i=1}^n R_{i\ell}^2=n(n+1)(2n+1)/6$ and $\sum_{i=1}^n R_{i\ell}R_{im} = n^3 \int \Pi_2 \diff \hat C_{n,I}$, we have
\begin{align*}
\hat \rho_{n,I} 
= 
1- \frac{6 \sum_{i=1}^n (R_{i\ell} - R_{im})^2}{n(n-1)(n+1)} 
&=
1- \frac{6 \sum_{i=1}^n R_{i\ell}^2 - 2R_{i\ell}R_{im}+R_{im}^2}{n(n-1)(n+1)} 
\\&=
1- \frac{2(2n+1)}{n-1} + \frac{12 n^2}{(n-1)(n+1)} \int \Pi_2 \diff \hat C_{n,I}
\\&=
-3 - \frac{6}{n-1} + \frac{12 n^2}{n^2-1} \Big\{ \int  \hat C_{n,I} \diff\Pi_2 + \frac1n \Big\}
\\&=
12 \int  \hat C_{n,I} \diff\Pi_2  - 3 + r_{n,I}^\rho,
\end{align*}
where
\begin{align*}
r_{n,I}^\rho = - \frac{6}{n-1} + \frac{12}{n^2-1}  \int  \hat C_{n,I} \diff\Pi_2 + \frac{12 n}{n^2-1}
&=
\frac{6}{n^2-1} \Big\{ -(n+1) + 2 \int  \hat C_{n,I} \diff\Pi_2 + 2n \Big\}
\\&=
\frac{6}{n^2-1} \Big\{ n - 1 + 2 \int  \hat C_{n,I} \diff\Pi_2 \Big\}.
\end{align*}
Using the crude bound $0 \le  \int  \hat C_{n,I} \diff\Pi_2 \le 1$, we obtain the bound
\begin{align*}
|r_{n,I}^\rho| \le \frac{6}{n^2-1} ( n + 1 ) = \frac{6}{n-1}.
\end{align*}

For Kendall's tau, using that $\mathrm{sgn}(u) = 2 \bm 1(u>0)-1$ for $u \ne 0$,
\begin{align*}
\hat \tau_{n,I} 
&=
\frac2{n(n-1)} \sum_{1 \le i < j \le n} \mathrm{sgn}(X_{i\ell}-X_{j\ell}) \mathrm{sgn}(X_{im}-X_{jm})
\\&=
\frac1{n(n-1)} \sum_{1 \le i \ne j \le n} \mathrm{sgn}(X_{i\ell}-X_{j\ell}) \mathrm{sgn}(X_{im}-X_{jm}) \\
&=
\frac1{n(n-1)} \sum_{1 \le i \ne j \le n} \mathrm{sgn}(R_{i\ell}-R_{j\ell}) \mathrm{sgn}( R_{im}- R_{jm}) \\
&=
\frac1{n(n-1)}   \sum_{1 \le i \ne j \le n} \Big\{ 4 \bm 1 (R_{i\ell}>R_{j\ell}, R_{im} >R_{jm}) - 2 \bm 1 (R_{i\ell}>R_{j\ell}) - 2 \bm 1 (R_{im}>R_{jm}) +1 \Big\}.
\end{align*}
Next, 
\[
\sum_{1 \le i \ne j \le n} \bm 1 (R_{i\ell}>R_{j\ell}) = \sum_{1 \le i \ne j \le n} \bm1 (i>j) = \frac{n(n-1)}2,
\]
and
\begin{align*}
\sum_{1 \le i \ne j \le n} \bm 1 (R_{i\ell}>R_{j\ell}, R_{im} >R_{jm})
=
\sum_{1 \le i, j \le n} \bm 1 (R_{i\ell}>R_{j\ell}, R_{im} >R_{jm})
&=
\sum_{1 \le i , j \le n} \bm 1 (\hat U_{i\ell}> \hat U_{j\ell}, \hat U_{im} > \hat U_{jm})
\\&=
- n + \sum_{1 \le i , j \le n} \bm 1 (\hat U_{i\ell} \ge \hat U_{j\ell}, \hat U_{im} \ge  \hat U_{jm}) 
\\&=
- n + n^2 \int \hat C_{n,I} \diff \hat C_{n,I},
\end{align*}
(the third equality being valid since there is no ties), which implies
\begin{align*}
\hat \tau_{n,I} 
=
\frac1{n(n-1)} \Big\{ - n + 4 n^2 \int \hat C_{n,I} \diff \hat C_{n,I} - 2 n(n-1) + n(n-1) \Big\}
&= \frac{n}{n-1} \Big\{ 4  \int \hat C_{n,I} \diff \hat C_{n,I}-1  \Big\}
\\&=
 4 \int \hat C_{n,I} \diff \hat C_{n,I}- 1 + r_{n,I}^\tau,
\end{align*}
where 
\[
r_{n,I}^\tau 
=
\frac1{n-1} \Big\{ 4 \int \hat C_{n,I} \diff \hat C_{n,I} - 1\Big\}.
\]
Hence, $|r_{n,I}^\tau| \le 4/(n-1)$.

\section{Proof of Lemma~\ref*{lem:cnctn}}
\label{sec:reduction-inverses}

For the sake of completeness, we repeat the lemma.

\begin{lemma-no-number} 
Fix $k \in \N_{\ge 2}$.
Assume that, for each $j \in \{1, \dots, d\}$, the $j$th first-order partial derivative $\dot C_j$ exists and is continuous on $V_{j} \cap W_k$.
Then, with probability one,
\begin{align} \label{eq:cnctn2}
\sup_{\bm u \in W_k} | \Cb_n(\bm u) - \tilde \Cb_n(\bm u)| \le \frac{k}{\sqrt n}.
\end{align}
\end{lemma-no-number}

\begin{proof}
The event 
$
\Omega' := \{\text{there exists $n\in\N, j \in \{1, \dots, d\}$ such that there are ties among}$ 
$U_{1j}, \dots, U_{nj}\}
$
has probability zero. We will show \eqref{eq:cnctn2} on the complement of that event.
Fix $\bm u \in W_k$ and let $I_{\bm u}$ denote the set of indexes $j$ for which $u_j <1 $; note that $|I_{\bm u}|  \le k$. For the ease of notation, we assume $|I_{\bm u}| = k$. Then,
$
\hat C_n(\bm u)  = \frac1n \sum_{i=1}^n \prod_{j\in I_{\bm u}} \bm1( \hat U_{i,j} \le u_j).
$
The product in the previous equation can be written as a telescopian sum
\begin{align*}
\prod_{j\in I_{\bm u}} \bm1( \hat U_{i,j} \le u_j)
&=
\prod_{j\in I_{\bm u}} \bm1( U_{i,j} \le G_{n,j}^-(u_j))
+ 
\Gamma_n(\bm u),
\end{align*}
where, writing $I_u=\{j_1, \dots, j_k\}$, 
\begin{multline*}
\Gamma_n(\bm u) 
= \sum_{\ell=1}^k \Big( \prod_{s=1}^{\ell-1} \bm 1(\hat U_{i,j_s} \le u_{j_s} )\Big) 
\Big( \prod_{s=\ell+1}^k \bm 1(U_{i,j_s} \le G_{n,j_s}^-(u_{j_s})) \Big) \\
\times  \Big\{ \bm 1(\hat U_{i,j_\ell} \le u_{j_\ell})  - \bm 1(U_{i,j_\ell} \le G_{n,j_\ell}^-(u_{j_\ell})) \Big\}.
\end{multline*}
Here,
\[
|\Gamma_n(\bm u)| 
\le 
\sum_{\ell=1}^k \big|\bm 1(\hat U_{i,j_\ell} \le u_{j_\ell})  - \bm 1(U_{i,j_\ell} \le G_{n,j_\ell}^-(u_{j_\ell})) \big|
\le
\sum_{\ell=1}^k \bm 1(U_{i,j_\ell} = G_{n,j_\ell}^-(u_{i,j_\ell})).
\]
As a consequence,
$
|\hat C_n(\bm u)-\tilde C_n(\bm u)|
\le 
\frac1n \sum_{i=1}^n \sum_{\ell=1}^k \bm 1(U_{i,j_\ell} = G_{n,j_\ell}^-(u_{i,j_\ell})) \le \frac{k}{n}
$
almost surely, because there are no ties with probability one. This implies the assertion. 
\end{proof}

\section*{Acknowledgments}
The authors are grateful to two unknown referees and an Associate Editor for their constructive comments that helped to improve the presentation substantially. The authors are grateful to Katharina Effertz for pointing out a possible improvement of a former version of Lemma \ref{lem:sn3-new} that allowed to get of rid of loglog-terms in the final result. 

\section*{Funding}
The first author was supported by the 
Deutsche For\-schungsgemeinschaft (DFG, German Research Foundation; Project-ID 520388526;  TRR 391:  Spatio-temporal Statistics for the Transition of Energy and Transport), 
which is greatly acknowledged.

\bibliographystyle{apalike}
\bibliography{biblio}
\end{document}